\documentclass[a4paper,11pt,reqno]{amsart}
\usepackage{amsmath}
\usepackage{amsfonts}
\usepackage[applemac]{inputenc}
\usepackage[T1]{fontenc}
\usepackage{graphicx}
\usepackage[mathscr]{euscript}
\usepackage{verbatim,color,lipsum}
\usepackage{amssymb,latexsym}
\usepackage{amsthm}
\usepackage{color}

\usepackage[lmargin=2.5 cm,rmargin=2.5 cm,tmargin=3.5cm,bmargin=2.5cm,paper=a4paper]{geometry}
\DeclareMathOperator{\curl}{curl}

\newtheorem{thm}{Theorem}[section]

\newtheorem{lem}[thm]{Lemma}

\newtheorem{assumption}[thm]{Assumption}
\newtheorem{definition}[thm]{Definition}
\newtheorem{lemma}[thm]{Lemma}
\newtheorem{proposition}[thm]{Proposition}
\newtheorem{corollary}[thm]{Corollary}

\theoremstyle{remark}
\newtheorem{rem}[thm]{Remark}
\newtheorem{notation}{Notation}[section]
\newcommand{\nb}{\nabla}
\newcommand{\kp}{\kappa}

\newcommand{\Ab}{\mathbf{A}}
\newcommand{\Cb}{\mathbb{C}}
\newcommand{\Fb}{\mathbf{F}}

\newcommand{\R}{\mathbb{R}}

\newcommand{\norm}[1]{\left\|#1\right\|}
\newcommand{\ddiv}{{\rm div\,}}

\numberwithin{equation}{section}

\title{The distribution of 3D superconductivity near the second critical field}
\author[A. Kachmar]{Ayman Kachmar}
\address{Lebanese University, Department of Mathematics, Hadath, Lebanon}
\email{ayman.kashmar@gmail.com}
 \author[M. Nasrallah]{Marwa Nasrallah}
 \address{Lebanese International University, Beirut, Lebanon \newline \& Lebanese University, faculty of Sciences, Section IV, Bekaa, Lebanon}
\email{marwa.nasrallah@liu.edu.lb}
\begin{document}
\maketitle 


\begin{abstract}
We study the minimizers of the Ginzburg-Landau energy functional with a constant magnetic field  in a three dimensional bounded domain. The functional depends on two positive parameters, the Ginzburg-Landau parameter and  the intensity of the applied magnetic field, and acts on complex valued functions  and vector fields.
We establish a formula for the distribution of the $L^2$-norm of the minimizing complex valued function (order parameter). The formula is valid in the regime where the  Ginzburg-Landau parameter is large and  the  applied magnetic field is close to the second critical field---the threshold value corresponding to the transition from the superconducting to the normal phase  in the bulk of the sample.  Earlier results are valid in $2D$ domains and  for the $L^4$-norm in $3D$ domains.
\end{abstract}

\section{Introduction}

In this paper, we derive a formula displaying the distribution of the density of the superconducting electron pairs (Cooper pairs) in a superconducting sample. Such a formula has been obtained in \cite{Ka-SIAM} when the sample occupies a cylindrical domain with an infinite height. The novelty here is that the sample is allowed to occupy any bounded three dimensional domain with a smooth boundary.

Our results are valid for type~II superconductors within the Ginzburg-Landau theory. In this theory, a superconducting sample is distinguished by a material parameter $\kappa>0$. $\kappa$ is called the Ginzburg-Landau parameter. When the sample is placed in a magnetic field, we will denote the intensity of the magnetic field by the positive parameter $H>0$. As $H$ varies, the state of superconductivity in the sample will undergo several phase transitions that we outline below:
\begin{itemize}
\item There is a first critical value $H_{C_1}>0$ such that, if $H<H_{C_1}$, the sample remains in a perfect superconducting state and repels the applied magnetic field.
\item There is a second critical value $H_{C_2}>H_{C_1}$ such that, if $H_{C_1}<H<H_{C_2}$, then the applied magnetic field penetrates the  sample in point defects and these point defects are in the normal (non-superconducting) state. The rest of the sample is in the superconducting state. The point defects are arranged along a lattice.
\item There is a third critical value $H_{C_3}>H_{C_2}$ such that, if $H_{C_2}<H<H_{C_3}$, then the bulk of the sample is in the normal state and the surface of the sample is in the superconducting state.
\item If $H>H_{C_3}$, all the sample is in the normal state. 
\end{itemize}
We refer the reader to the book of de\,Gennes \cite{dGe} for the physical background. Using the Ginzburg-Landau model and rigorous mathematical methods, the critical values (fields) $H_{C_1}$, $H_{C_2}$ and $H_{C_3}$ are identified in the large $\kappa$ regime. For samples occupying infinite cylindrical domains, we refer the reader to the papers \cite{AfSe, Alm-CMP,  CR, CR2, FK-am, Pan, SS02} and the two monographs \cite{FH-b, SS-b}. For general three dimensional domains, we refer the reader to the papers \cite{BJS, FH-3D, FK-cpde, FK-jmpa, HeMo, Ka-JFA, Lu-Pan}. The value $H_{C_2}$ is called the second critical field. Existing results suggest that $H_{C_2}\sim\kappa$ as $\kappa\to\infty$, for samples with Ginzburg-Landau parameter $\kappa$ (cf. \cite{AfSe, FK-am, Pan}).

 Suppose that the superconducting sample occupies a domain $\Omega\subset\R^3$. The state of the superconductivity is described using a complex-valued function $\psi:\Omega\to\mathbb C$ and a vector field $\Ab:\Omega\to\R^3$. The function $\psi$ is called the Ginzburg-Landau parameter and the vector field $\Ab$ is called the magnetic potential. The quantity $|\psi|^2$ measures the density of the superconducting electron pairs (Cooper pairs) hence when $\psi(x)\approx0$ the sample is in the normal state at $x$.  At equilibrium, the configuration $(\psi,\Ab)$ minimizes the Ginzburg-Landau energy.  

If the region $\Omega$ is an infinite cylinder with cross section $U\subset\R^2$ and the applied magnetic field is parallel to the cylinder's axis, then $\psi$ and $\Ab$ can be reduced to functions defined on $U$. In this case, under the assumptions
\begin{equation}\label{eq:assum-2D}
\kappa\to\infty\quad{\rm and}\quad \kappa^{-1/2}\ll1-\frac{H}\kappa\ll 1\,,
\end{equation}
the density $|\psi|^2$ satisfies (cf. \cite{Ka-SIAM})
\begin{equation}\label{eq:Ka-2D-main}
\int_U|\psi|^2\,dx=-E_{\rm Ab}|U|\,[\kappa-H]^2+o([\kappa-H]^2)\,.\end{equation}
Here $E_{\rm Ab}\in[-\frac12,0)$ is a universal constant, called the Abrikosov constant and  will be defined later. 

In \eqref{eq:assum-2D}, we use the following notation. For positive functions $a(\kappa)$ and $b(\kappa)$,  $a(\kappa)\ll b(\kappa)$ means that there exists $\delta(\kappa)$ such that $\displaystyle\lim_{\kappa\to\infty}\delta(\kappa)=0$ and $a(\kappa)=\delta(\kappa)b(\kappa)$.

Note that the assumption in \eqref{eq:assum-2D} corresponds to the regime close to the second critical field  and is the optimal assumption needed for \eqref{eq:Ka-2D-main} to be valid (cf. \cite{FK-am, Ka-SIAM}).

The aim of this paper is to obtain an analogue of the formula in \eqref{eq:Ka-2D-main} when the domain $\Omega$ is a general bounded domain of $\R^3$ with a smooth boundary. This will improve and complete the results in \cite{FK-cpde, FK-jmpa}.

Hereafter, we suppose that $\Omega\subset\R^3$ is {\bf open}, {\bf bounded}, has a {\bf finite} number of {\bf connected} components and with a {\bf smooth boundary}.  For every configuration 
$(\psi,\Ab)\in H^1(\Omega;\Cb)\times H^1_{\rm loc}(\R^3;\R^3)$, we define the Ginzburg-Landau energy of $(\psi,\Ab)$ as follows
\begin{multline}\label{GL-Energy}
\mathcal{E}^{3D}(\psi,{\bf A})=\int_{\Omega}\left[ |(\nabla -i\kappa H{\bf A})\psi|^{2}-\kappa^{2}|\psi|^{2}+\frac{\kappa^{2}}{2}|\psi|^{4}                                  \right]dx+\kappa^{2}H^{2}\int_{\R^{3}}|{\curl} {\bf A}-\beta|^{2}dx.
\end{multline}
Here, as explained earlier,  $\kappa$ and $H$ are two positive parameters, and  $\beta=(0,0,1) $ is the profile and direction of the (constant) applied magnetic field.

Let us introduce the space  $ \dot{H}^{1}_{\ddiv , {\bf F}}(\R^{3})$ of vector fields defined as follows
\begin{equation}\label{mag-Sob-space}
\dot{H}^{1}_{\ddiv, {\bf F}}(\R^{3})=\Big\{{\bf A}:\R^3\to\R^3\quad:\quad \ddiv {\bf A}=0, \quad {\rm and}\quad {\bf A-F}\in \dot{H}^{1}(\R^{3})\Big\}\,,
\end{equation}
where ${\bf F}$ is the following magnetic potential 
\begin{equation}\label{MP-F}
{\bf F}(x)= (-x_{2}/2,x_{1}/2,0),\quad \forall \,x=(x_{1},x_{2},x_{3})\in \R^{3},
\end{equation}
and
the space  $ \dot{H}^{1}(\R^{3})$ is the homogeneous Sobolev space, i.e. the closure of $C^{\infty}_{c}(\R^{3})$ under the norm $u\mapsto \norm{u}_{\dot{H}^{1}(\R^{3})}:= \norm{\nb u}_{L^{2}(\R^{3})}$.

The energy in \eqref{GL-Energy} will be minimized over the space $H^1(\Omega;\mathbb C)\times \dot{H}^{1}_{\ddiv, {\bf F}}(\R^{3})$. Actually, this is the natural `energy' space for the functional in \eqref{GL-Energy}, see \cite{FH-b}. We thereby introduce the following ground state energy
\begin{equation}\label{eq:gse**}
E_{g.st}(\kp,H)=\inf \{\mathcal{E}^{3D}(\psi,{\bf A})\quad:\quad (\psi,{\bf A})\in H^{1}(\Omega;\mathbb{C})\times \dot{H}^{1}_{\ddiv,{\bf F}}(\R^{3})\}\,.
\end{equation}
For a given $\kappa$ and $H$, we will call a minimizer of the functional \eqref{GL-Energy} a configuration $(\psi,\Ab)\in H^{1}(\Omega;\mathbb{C})\times \dot{H}^{1}_{\ddiv,{\bf F}}(\R^{3})$ satisfying $\mathcal{E}^{3D}(\psi,{\bf A})=E_{g.st}(\kp,H)$. Obviously, such a configuration will depend on $\kappa$ and $H$. To emphasize this dependence, we will denote such minimizers by $(\psi,\Ab)_{\kappa,H}$. 

Note that a minimizer $(\psi,\Ab)_{\kappa,H}$ is a {\it critical point} of the functional in \eqref{GL-Energy}, i.e.
$$\forall~(\phi,a)\in H^1(\Omega;\Cb)\times C_c^\infty(\R^3;\R^3)\,,~\frac{d}{dt}\mathcal{E}^{3D}(\psi+t\phi,\Ab)\Big|_{t=0}=0~{\rm and}~ \frac{d}{dt}\mathcal{E}^{3D}(\psi,\Ab+ta)\Big|_{t=0}=0\,.$$
More precisely,  a critical point $(\psi,{\bf A})\in H^{1}(\Omega;\mathbb{C})\times H^{1}_{\ddiv,{\bf F}}(\R^{3})$  is a weak solution of the Ginzburg-Landau equations,
\begin{equation}\label{E-GL}\left\{
\begin{array}{lcl}
-(\nb-i\kp H{\bf A})^{2}\psi=\kp^{2}(1-|\psi|^{2})\psi& {\rm in} &\Omega\\
\curl^{2}{\bf A}=-\dfrac{1}{\kp H}{\rm Im}(\bar{\psi}(\nabla -i\kp H{\bf A})\psi){\bf 1}_{\Omega}&{\rm in}&\R^{3}\\
\nu\cdot(\nb-i\kp H{\bf A})\psi=0&{\rm on}&\partial\Omega.
\end{array}
\right.
\end{equation}
where ${\bf 1}_{\Omega}$ is the characteristic function of the domain $\Omega$, and $\nu$ is the unit interior normal vector of $\partial\Omega$.

Minimizers of the functional in \eqref{GL-Energy} are studied in \cite{FK-cpde, FK-jmpa}. Under the assumption in \eqref{eq:assum-2D}, if $(\psi,\Ab)_{\kappa,H}$ is a minimizer of the functional in \eqref{GL-Energy}, then
\begin{equation}\label{eq:l4-FK-main}
\int_\Omega|\psi|^4\,dx=-2E_{\rm Ab}|\Omega|\,\left(1-\frac{H}\kappa\right)^2+
o\left(\left(1-\frac{H}\kappa\right)^2\right)\,.\end{equation}
We will improve this formula in Theorem~\ref{thm:op-l4} below. 
We will work under the following assumption:

\begin{assumption}\label{assump}~
\begin{itemize}
\item $\alpha:\R_+\to\R_+$ and $\beta:\R_+\to\R_+$ are two functions satisfying 
$$\lim_{\kappa\to\infty}\alpha(\kappa)=\infty\,,\quad\lim_{\kappa\to\infty}\beta(\kappa)=0\quad{\rm and}\quad \alpha(\kappa)\leq \beta(\kappa)\kappa^{1/2} {~\rm in~a~neighborhood~of~}\infty\,.$$
\item $\kappa>0$ and $H>0$ satisfy 
$\displaystyle\alpha(\kappa)\kappa^{-1/2}\leq 1-\frac{H}\kappa\leq \beta(\kappa)$.
\end{itemize}
\end{assumption}
In this paper, we will prove the following theorem (compare with \eqref{eq:l4-FK-main}):
\begin{thm}\label{thm:op-l4}{\bf[Sharp bound in $L^4$-norm]}
 
There exist $\kappa_0>0$ and a function ${\rm err}:[\kappa_0,\infty)\to(0,\infty)$ such that:
\begin{itemize}
\item $\displaystyle\lim_{\kappa\to\infty}{\rm err}(\kappa)=0$\,;
\item the following inequality holds
\begin{equation}\label{eq:l4-KN-main}
\dfrac1{|Q_{\kp}|}\int_{Q_{\kp}}{|\psi|^{4}}dx\leq -2E_{\rm Ab}\Big(1-\dfrac{H}{\kp}\Big)^{2}+\Big(1-\dfrac{H}{\kp}\Big)^{2}{\rm err}(\kappa)\,,
\end{equation}
where 
\begin{itemize}
\item $E_{\rm Ab}$ is the Abrikosov constant introduced below in Theorem~\ref{thm:Ab}\,;
\item $\kappa\geq\kappa_0$ and $(\kappa,H)$ satisfy Assumption~\ref{assump}\,;
\item $(\psi,\Ab)$ is a solution of \eqref{E-GL}\,;
\item $Q_{\kappa}$ is any cube of side length $\kappa^{-1/2}$ and satisfying
$
\overline{Q_{\kappa}}\subset \{{\rm dist}(x,\partial\Omega)>2\kp^{-1/2}\}.
$
\end{itemize}
\end{itemize}
\end{thm}

Note that the conclusion in Theorem~\ref{thm:op-l4} has been known  in the following cases:
\begin{itemize}
\item when $Q_\kappa$ is replaced by the whole domain $\Omega$ but without specifying the (sharp) constant $E_{\rm Ab}$ (cf. \cite{Alm-3D})\,;
\item when $Q_\kappa$ is replaced by any open subset $D\subset \overline{D}\subset\Omega$ and with a smooth boundary (cf. \cite{FK-cpde}). 
\end{itemize}

In light of \eqref{eq:l4-FK-main}, we observe that the constant $E_{\rm Ab}$ in \eqref{eq:l4-KN-main} is optimal. 

Let us point out that the derivation in \cite{FK-cpde, FK-jmpa} of the upper bound in \eqref{eq:l4-FK-main}  relies on the estimate in \cite{Alm-3D} to control the error terms. However, the proof we give to Theorem~\ref{thm:op-l4} does not use ingredients from \cite{Alm-3D} but instead uses Theorem~\ref{NF} in this paper, which displays a new formulation of the Abrikosov constant  in terms of a non-linear eigenvalue problem.

Our next result is an asymptotic formula of the $L^2$-distribution of the minimizing order parameters.

\begin{thm}\label{thm:op-l2-main}{\bf [Distribution of the density]}
Let $D\subset\Omega$ be an open set such that $|\partial D|=0$. 
Suppose that $H$ is a function of $\kappa$ satisfying
\begin{equation}\label{eq:assum-3D-H}
H\leq \kappa\quad{\rm and}\quad \kappa^{-9/26}\ll1-\frac{H}\kappa\ll1\,.\end{equation}

If $(\psi,\Ab)_{\kappa,H}$ is a minimizer of the functional in \eqref{GL-Energy}, then as $\kappa\to\infty$,
\begin{equation}\label{eq:l2-3D-main}
\frac{1}{|D|}\int_D|\psi|^2\,dx=-2E_{\rm Ab}\left(1-\frac{H}\kappa\right)+o\left(1-\frac{H}\kappa\right)\,.
\end{equation}
Here $E_{\rm Ab}\in[-\frac12,0)$ is the universal constant defined in Theorem~\ref{thm:Ab} below.
\end{thm}

Note that the conclusion in Theorem~\ref{thm:op-l2-main}  is consistent with the formula in \eqref{eq:Ka-2D-main} but is valid under the more restrictive assumption in \eqref{eq:assum-3D-H}.
One reason that prevented us of proving \eqref{eq:l2-3D-main} under the assumption in \eqref{eq:assum-2D} is the lack of the upper bound 
\begin{equation}\label{eq:2D-up-FK}
\|\psi\|_{L^\infty(\Omega_{\kappa})}\leq C\left|1-\frac{H}{\kappa}\right|^{1/2}\quad\big(\Omega_\kappa=\{x\in\Omega~:~{\rm dist}(x,\partial\Omega)\gg\kappa^{-1}\}\big)\,. 
\end{equation}
 This upper bound is shown to hold in $2D$ domains (cf. \cite{FK-am}). Since we were not able to prove \eqref{eq:2D-up-FK} in  $3D$ domains, we used the estimate in Theorem~\ref{thm:op-l4}
as a substitute.  The price we paid is the restrictive assumption in \eqref{eq:assum-3D-H}. The technical reasons that led us to the assumption in \eqref{eq:assum-3D-H} are explained in Remark~\ref{rem:par}.

The rest of the paper is devoted to the proof of Theorems~\ref{thm:op-l4} and \ref{thm:op-l2-main}. It is organized as follows:
\begin{itemize}
\item Section~\ref{sec:LimEn} reviews various limiting energies studied in \cite{FK-cpde} and concludes with the proof of  Theorem~\ref{NF}. Theorem~\ref{NF} is new and not among the results in \cite{FK-cpde}.
\item Section~\ref{sec:l4} is devoted to the proof of Theorem~\ref{thm:op-l4}. It uses Theorem~\ref{NF} as a key ingredient.
\item Section~\ref{sec:en} establishes asymptotics of the Ginzburg-Landau energy in cubes with small lengths. The main conclusion here is summarized in Corollary~\ref{corol:lb-op-l4}. The assumption in \eqref{eq:assum-3D-H} is needed in this section.
\item Section~\ref{sec:l2} finishes the proof of Theorem~\ref{thm:op-l2-main}. We prove an energy asymptotics for the density in cubes with small lengths as well, see Corollary~\ref{corol:Pi-Ab}.   
\end{itemize}

\subsection*{Remark on the notation}

The parameters $\kappa$ and $H$ are allowed to vary in such a manner that $H/\kappa\in[c_1,c_2]$, where $0<c_1<c_2$ are fixed constants. Whenever the letter $C$ appears, it denotes a positive constant that is independent of $\kappa$ and $H$. Such a constant may depend on the domain $\Omega$, the constants $c_1,c_2$, etc. The value of $C$ might change from one formula to another.

In the proofs, the notaion $o(1)$ stands for an expression that depends on $\kappa$ and $H$ such that $o(1)\to0$ as $\kappa\to\infty$. However, this expression is independent of the choice of a minimizing/critical configuration $(\psi,\Ab)_{\kappa,H}$ of the functional in \eqref{GL-Energy}, but it depends on the constants $c_1,c_2$, the domain $\Omega$, etc. Sometimes we do local arguments in, say, a ball or a square of cener $x_0$ and radius $\ell$. In such arguments, the quantity $o(1)$ is {\it independent of the center} $x_0$ but do depend on the radius $\ell$. 

Finally, by writing $a(\kappa)\approx b(\kappa)$, we mean that the {\it positive}  functions $a(\kappa)/b(\kappa)$ and $b(\kappa)/a(\kappa)$ are bounded in a neighborhood of $\kappa=\infty$. In particular, our assumption on $\kappa$ and $H$ can be expressed as $H\approx \kappa$.

\section{Limiting energies}\label{sec:LimEn}
\subsection{Two-dimensional limiting energy}
\subsubsection{Reduced Ginzburg-Landau functional and thermodynamic limit}
Let $b>0$ and $D$ be an open subset in $\R^{2}$. We define the following reduced Ginzburg-Landau functional,
\begin{equation}\label{R-GL-2}
H^1(D)\ni u\mapsto G_{b,D}(u)=\int_{D}\Big( b |(\nb-i{\bf A}_{0})u|^{2}-|u|^{2}+\dfrac{1}{2}|u|^{4}\Big)dx\,,
\end{equation}
where
\begin{equation}\label{A0}
{\bf A}_{0}(x_{1},x_{2})=\dfrac{1}{2}(-x_{2},x_{1}),\qquad \Big( x=(x_{1},x_{2})\in\R^{2}\Big)\,.
\end{equation}
Given $R>0$, we denote by $K_{R}=(-R/2,R/2)^{2}$ the square of side length $R$ and center $0$. Let us introduce the following ground state energy
\begin{equation}\label{eq:en-m0}
m_{0}(b,R)=\inf_{u\in H^{1}_{0}(K_{R};\mathbb{C})}G_{b,K_{R}}(u)
\end{equation}
It is proved in \cite{AfSe, FK-cpde, SS02} that, for all $b\geq 0$, there exists $g(b)\in [-\frac12,0]$ such that
\begin{equation}\label{eq:g(b)}
g(b)= \lim_{R\rightarrow \infty}\dfrac{m_{0}(b,R)}{R^{2}}\,,
\end{equation}
and that the  function $[0,\infty)\ni b\mapsto g(b)\in[-1/2,0]$ is continuous, non-decreasing, $g(0)=-\frac12$  and $g(b)= 0$ for all $b\geq 1$. Moreover, there exists a universal constant $\alpha \in(0,1/2)$ such that, for all $b\in [0,1]$
\begin{equation}\label{bnd-on-g}
\alpha (b-1)^{2}\leq |g(b)|\leq \dfrac{1}{2}(b-1)^{2}.
\end{equation}
Also, for all $R\geq 1$ and  $b\in[0,1]$, it holds the estimate
\begin{equation}\label{eq:m0gb}
g(b)\leq \dfrac{m_{0}(b,R)}{R^{2}}\leq g(b)+\dfrac{C}{R}.
\end{equation}

\subsection{The $2D$ periodic Schr\"{o}dinger  operator with constant magnetic field}
Let $R>0$ and $K_{R}=(-R/2,R/2)\times(-R/2,R/2)$. In this section we assume that
\[
R^{2}\in 2\pi \mathbb{N}.
\]
We introduce the following space
\begin{multline}\label{ER}
E_{R}= \Bigg\{ u\in H^{1}_{\rm loc }(\R^{2};\mathbb{C})~:~ u(x_{1}+R,x_{2})=e^{iRx_{2}/2}u(x_{1},x_{2}), \\
u(x_{1},x_{2}+R)=e^{-ix_{1}/2}u(x_{1},x_{2}), \qquad (x_{1},x_{2})\in \R^{2}\Bigg\}\,.
\end{multline}
Recall the magnetic potential ${\bf A}_{0}$ in \eqref{A0}. Consider the operator
\begin{equation}\label{eq:op-2D}
P_{R}^{\rm 2D}=-(\nb -i{\bf A}_{0})^{2}
\end{equation}
with form domain $E_{R}$ introduced in \eqref{ER}. More precisely, $P_{R}^{\rm 2D}$ is the self-adjoint realization associated with the closed quadratic form
\[
E_{R}\ni f\mapsto Q_{R}^{\rm 2D}(f)= \norm{(\nb -i{\bf A}_{0})f}^{2}_{L^{2}(K_{R})}.
\]
The operator $P_{R}^{\rm 2D}$ has a compact resolvent. We denote  by $\big\{\mu_{j}(P_{R}^{\rm 2D})\big\}_{j\geq 1}$ the increasing sequence of its eigenvalues. The following proposition may be classical in the spectral theory of Schrodinger operators, but we refer to \cite{AfSe} or \cite{Alm-CMP} for a simple proof.
\begin{proposition}\label{prop:op-2D}
The operator $P_{R}^{\rm 2D}$ has the following properties:\begin{enumerate}
\item $\mu_{1}(P_{R}^{\rm 2D}) =1$, and   $\mu_{1}(P_{R}^{\rm 2D}) = 3$.
\item The space $L_{R}={\rm Ker}(P_{R}^{\rm 2D}-1) $ is finite dimensional and ${\rm dim}{L_{R}}=\dfrac{R^{2}}{2\pi}.$
\end{enumerate}
Consequently, denoting by $\Pi_{1}$ the orthogonal projection on the space $L_{R}$ in $L^{2}(K_{R})$ and by $\Pi_{2}={\rm Id}-\Pi_{1}$, we have for all $f\in D(P_{R}^{\rm 2D})$,
\begin{equation}\label{eq:gap}
\langle  P_{R}^{\rm 2D}\Pi_{2}f,\Pi_{2}^{\rm 2D}f       \rangle_{L^{2}(K_{R})} \geq  3\norm{\Pi_{2}f}_{L^{2}(K_{R})}^{2}.
\end{equation}
\end{proposition}
The next Lemma is a consequence of the existence of the spectral gap between the first two eigenvalues of $P_{R}^{\rm 2D}$. It is proved in \cite[Lemma~2.8]{FK-am}.
\begin{lem}\label{lem:op-2D}
Let $p\geq 2$.There exists a constant $C_{p}>0$ such that for any $\gamma \in (0,1/2)$, and $u\in D( P^{\rm 2D}_{R})$ satisfying
\begin{equation}\label{Asp-gam}
Q^{\rm 2D}_{R}(f)-(1+\gamma) \norm{f}^2_{L^{2}(K_{R})}\leq 0
\end{equation}
the following estimate holds:
\begin{equation}
\norm{u-\Pi_{1}u}_{L^{p}(K_{R})}\leq C_{p}\sqrt{\gamma}\norm{u}_{L^{2}(K_{R})}.
\end{equation}
Here $\Pi_{1}$ is the projection on the space $L_{R}$.
\end{lem}

\subsection{The Abrikosov energy}
We introduce the following energy functional (the Abrikosov energy):
\[
F_{R}(v)=\int_{K_{R}}\Big(  \frac{1}{2}|v|^{4}-|v|^{2}                          \Big)dx\,.
\]
The energy $F_{R}$ will be minimized on the space $L_{R}$, the (finite dimensional) eigenspace of the first eigenvalue of the periodic operator $P_{R}^{\rm 2D}$,
\[
L_{R}=\{u\in E_{R}\quad :\quad P_{R}^{\rm 2D}u=u\}\,.
\]
For all $R>0$, let
\begin{equation}\label{cR}
c(R)= \min \Big\{F_{R}(u)\quad:\quad u\in L_{R}\Big\}.
\end{equation}

The following theorem is proved in \cite{AfSe, FK-cpde}:

\begin{thm}\label{thm:Ab}
There exists a constant $E_{\rm Ab}\in [-1/2,0[$ such that
\[
E_{\rm Ab}= \lim_{\stackrel {R\rightarrow \infty}{R^{2}/2\pi\in \mathbb{N}}}\dfrac{c(R)}{R^{2}}=\lim_{b\to1_-}\frac{g(b)}{(b-1)^2}.
\]
\end{thm}

We collect one more estimate from \cite[Prop.~3.1~\&~Thm.~3.5]{Ka-SIAM}. There exist two constants $C>0$ and $\epsilon_0\in(0,1)$ such that, for all $b\in(1-\epsilon_0,1)$ and $R\geq 2$,
\begin{equation}\label{eq:Ka}
m_0(b,R)\leq (1-b)^2c(R)+C(1-b)R\,.\end{equation}

\subsection{Three-dimensional limiting energy}
Let $b>0$, $\mathcal D$ be an open subset in $\R^{3}$ and  
\begin{equation}\label{R-GL-3}
\forall~u\in H^1(\mathcal D)\,,\quad F_{b,\mathcal D}(u)=\int_{\mathcal D}\Big( b |(\nb-i{\bf F})u|^{2}-|u|^{2}+\dfrac{1}{2}|u|^{4}\Big)dx\,,
\end{equation}
where ${\bf F}$ is the magnetic potential introduced in \eqref{MP-F}. For all $R>0$, we denote by
 $Q_{R}=K_{R}\times(-R/2,R/2)$ and
\begin{equation}\label{M0}
M_{0}(b,R)=\inf_{u\in H^{1}_{0}(Q_{R};\mathbb{C})}F_{b,Q_{R}}(u).
\end{equation}
The next lemma displays the connection between the two and three  dimensional ground state energies, $m_{0}(b,R)$ and  $M_{0}(b,R)$. It is taken from \cite[Theorem~2.14]{FK-cpde}.
\begin{lem}\label{Lem1}
There exists a universal constant $C>0$ such that, for all $b\geq 0$ and $R>0$, we have
\begin{equation}\label{Eq:m0M0}
Rm_{0}(b,R)\leq M_{0}(b,R)\leq (R-2)m_{0}(b,R)+C.
\end{equation}
\end{lem}
Combining \eqref{eq:m0gb} and \eqref{Eq:m0M0}, we deduce the following lemma.
\begin{lem}\label{Lem2}
There exists a universal constant $C>0$ such that for all $R\geq 1$ and $b>0$,
\[
g(b)\leq \dfrac{M_{0}(b,R)}{R^{3}}\leq \dfrac{R-2}{R}g(b )+\dfrac{C}{R}.
\]
\end{lem}

As a consequence of Lemma~\ref{Lem2}, we may prove:

\begin{lem}\label{Lem:L4}
There exists a constant $C>0$,  such that, if  $b\in(0,1]$, $R>1$ and $v_{b,R}$ is a minimizer of $F_{b,Q_{R}} $ (i.e. $F_{b,Q_R}(v_{n,R})=M_0(b,R)$), then,
\begin{equation}
-2 R^{2}(R-2)g(b)-{C}R^{2}\leq \int_{Q_{R}}|v_{b,R}|^{4}dx\leq -2 R^{3}g(b )\,.\end{equation}
\end{lem}
\begin{proof}
The minimizer satisfies the following equation
\[
-b(\nb-i {\bf F})^{2} v_{b,R}=(1-|v_{b,R}|^{2})v_{b,R},
\]
with Dirichlet boundary conditions on the boundary of $Q_{R}$.

Multiplying the above equation by $\overline{v_{b,R}}$, integrating over $Q_{R} $ and performing an integration by parts, it follows that
\[
M_{0}(b,R)=-\dfrac{1}{2}\int_{Q_{R}}|v_{b,R}|^{4}dx.
\]
Now applying Lemma~\ref{Lem2} finishes the proof of Lemma~\ref{Lem:L4}.
\end{proof}

Now we establish a link between the ground state energy in \eqref{M0} and a non-linear eigenvalue problem. Such a relationship has been discovered in \cite{Ka-nf} in the two dimensional setting.

We define the linear functional
\begin{equation}\label{R-GL-3-lin}
F^{\rm lin}_{b,D}(u)=\int_{D}\Big( b |(\nb-i{\bf F})u|^{2}-|u|^{2}\Big)dx\,.
\end{equation}

We will minimize this functional in the space of functions satisfying
\[
\int_{Q_{R}}|u|^{4}dx=1.
\]

That way, we are led to introduce  the following ground state energy
\begin{equation}\label{matcalM}
\mathcal{M}_{0}(b,R)=\inf \Bigg\{     \dfrac{F^{\rm lin}_{b,D}(u)}  {\Big(\int_{Q_{R}}|u|^{4}dx\Big)^{1/2}}                 \quad:\quad u\in H^{1}_{0}(Q_{R})\setminus\{0\}\Bigg\}
\end{equation}
We aim to prove that
\begin{equation}
\displaystyle\lim_{R\rightarrow \infty}\dfrac{\mathcal{M}_{0}(b,R)}{R^{3/2}}= g_{\rm new}(b),
\end{equation}
where
\[
g_{\rm new}(b)=-\sqrt{-2g(b)}.
\]

Actually, it holds:

\begin{thm}\label{NF}
Let $b\in(0,1)$. There exist two constants $C>0$ and $R_0>1$ such that, for all $R\geq R_{0}$,
\begin{equation}
-(-2g(b))^{1/2}\leq  \dfrac{\mathcal{M}_{0}(b,R)}{R^{3/2}}\leq-\Big(\big(1-\dfrac{C}{R}\big)(-2g(b) )\Big)^{1/2}+\dfrac{C}{R}(-2g(b ))^{-1/2}.
\end{equation}
\end{thm}
In light of Theorem~\ref{NF}, we infer that
\[
g(b)=-\lim_{R\rightarrow \infty} \dfrac{1}{2}\Big(\displaystyle \lim_{R\rightarrow \infty}\dfrac{\mathcal{M}_{0}(b,R)}{R^{3/2}} \Big)^{2}.
\]

\begin{proof}[Proof of Theorem~\ref{NF}]~

{\bf Upper bound:}~

We will prove the following inequality
\begin{equation}\label{UB}
\mathcal{M}_0(b,R)\leq  -2({R-2}){R}^{1/2}(-2g(b ))^{1/2}+CR^{1/2}(-2g(b ))^{-1/2}
\end{equation}
valid for some universal constant $C$, for all $b\in(0,1)$ and $R$ sufficiently large.

Let $v_{b,R}$ be a minimizer of ${M}_{0}(b,R)$ for the Dirichlet boundary condition. Using the definition of $\mathcal{M}_0(b,R)$, we may write
\begin{align}
\nonumber F_{b,Q_{R}}(v_{b,R})&= M_{0}(b,R)\\
&\geq  \mathcal{M}_{0}(b,R)\left(\int_{Q_{R}}|v_{b,R}|^{4}dx\right)^{1/2}+\dfrac{1}{2}\int_{Q_{R}}|v_{b,R}|^{4}dx\,.
  \end{align}
By Lemma~\ref{Lem:L4}, we get for $R$ sufficiently large
$$
M_{0}(b,R)
 \geq   \mathcal{M}_{0}(b,R) \Big( -2 R^{2}(R-2)g(b)-{C}R^{2}\Big)_{+}^{1/2}+\dfrac{1}{2}\Big(-2 R^{2}(R-2)g(b)-{C}R^{2}\Big)\,.
$$
 We use Lemma \ref{Lem2} to estimate $M_{0}(b,R)$ from above. This finishes the proof of the upper bound in \eqref{UB}.

{\bf Lower bound:}

We will prove that for all $b\in(0,1)$ and $R>1$,
\begin{equation}\label{LB}
\mathcal{M}_{0}(b,R)\geq -R^{3/2}(-2g(b))^{1/2}.
\end{equation}
Let $ w_{b,R}$ be a minimizer of $\mathcal{M}_{0}(b,R) $. Let us normalize $w_{b,R}$ as follows
\[
w_{b,R}^{\ast}= \dfrac{  R^{3/4}(-2g(b ))^{1/4}}{\norm{w_{b,R}}_{L^{4}(Q_{R})}}w_{b,R}\,.
\]
The $L^{4}$- norm of $w_{b,R}$ satisfies
\[
\norm{w_{b,R}}_{L^{4}(Q_{R})}= {  R^{3/4}(-2g(b ))^{1/4}}.
\]
By definition of  $\mathcal{M}_{0}(b,R)$, we see that
\begin{equation}\label{Flin-F}
\mathcal{M}_{0}(b,R)= \dfrac{{F}^{\rm lin}_{b,Q_{R}}(w_{b,R})}{{\norm{w_{b,R}}^{2}_{L^{4}(Q_{R})}}}=R^{-3/2}(-2g(b ))^{-1/2} F^{\rm lin}_{b,Q_{R}}(w_{b,R}^{\ast}).
\end{equation}
We write
\begin{align}
\nonumber{F}^{\rm lin}_{b,Q_{R}}(w_{b,R}^{\ast})&={F}_{b,Q_{R}}(w_{b,R}^{\ast}) -\dfrac{1}{2} \int_{Q_{R}}|w_{b,R}^{\ast}|^{4}dx\\
\nonumber&= {F}_{b,Q_{R}}(w_{b,R}^{\ast})+R^{3}g(b)\\
\nonumber&\geq M_{0}(b,R)+R^{3}g(b)\\
\nonumber&\geq 2R^{3}g(b).
\end{align}
Note that in the last inequality, we used Lemma~\ref{Lem2} to write an upper bound for $M_0(b,R)$.

Now, inserting the inequality $\nonumber{F}^{\rm lin}_{b,Q_{R}}(w_{b,R}^{\ast})\geq 2R^3g(b)$  in \eqref{Flin-F}, we obtain \eqref{LB}.
\end{proof}

\section{Proof of Theorem~\ref{thm:op-l4}}\label{sec:l4}

In the sequel, we will work with the following  local energy
\begin{equation}\label{E0}
\mathcal{E}_{0}(\psi,{\bf A};D)= \int_{D}\left(  |(\nb-i\kp H{\bf A})\psi|^{2}-\kp^{2}|\psi|^{2}+\frac{\kp^{2}}{2}|\psi|^{4}                             \right)dx\quad (D\subset\Omega)\,.
\end{equation}

We collect various a priori estimates that are useful in the proof of Theorem~\ref{thm:op-l4} (cf. \cite[Chapter~10]{FH-b}). 

\begin{lem}\label{Lem:est}
If $(\psi,{\bf A})$ is a solution of \eqref{E-GL}, then
\begin{equation}
\norm{\psi}_{\infty}\leq 1,
\end{equation}
\begin{equation}\label{grad-est}
\norm{(\nb -i\kp H{\bf A})\psi}_{C^{1}(\overline\Omega)}\leq C_{1}\sqrt{\kp H}\norm{\psi}_{L^{\infty}}\,,
\end{equation}
and
\begin{equation}\label{curl-est}
\norm{\curl(\Ab-\Fb)}_{C^{1}(\overline\Omega)}\leq \frac{C_{1}}{H}\norm{\psi}_{L^{\infty}}\norm{\psi}_{L^2(\Omega)}\,.
\end{equation}
\end{lem}

\begin{lem}\label{lem:gauge}
There exist positive constants $C$ and $\kappa_{0}$ such that if
\[
\kappa\geq\kappa_0\,,\quad\Lambda_{\rm min}\leq \frac{H}{\kappa}\leq \Lambda_{\rm max},
\]
and  $(\psi,{\bf A})$ is a solution of \eqref{E-GL}, then the following is true. 

Let $\ell\in(0,1)$ and $Q_{\ell}\subset\Omega$ be a cube of side length $\ell$, then there exists a function $\phi\in C^\infty(\overline{Q_\ell})$ such that, for all $x\in Q_{\ell}$, we have
\begin{equation}\label{gaugeAF}
|{\bf A}(x)-{\bf F}(x)-\phi(x)|\leq C\dfrac{\lambda^{1/6}}{\kp}\ell,
\end{equation}
where
\[
\lambda=\max\left(\frac1\kappa,\left(1-\frac{H}{\kappa}\right)^2\right).
\]
\end{lem}
\begin{proof}
In \cite[Corollary~4.4]{FK-jmpa}, it is proved that $\|\Ab-\Fb\|_{C^{1,1/2}(\overline{\Omega})}\leq C\kappa^{-1}\lambda^{1/6}$. The conclusion in 
Lemma~\ref{lem:gauge} follows by taking $\phi(x)=(\Ab(x_0)-\Fb(x_0))\cdot(x-x_0)$ where $x_0$ is the center of the square $Q_\ell$.
\end{proof}

\begin{proof}[Proof of Theorem~\ref{thm:op-l4}]
Let $\sigma\in(0,1)$ and $Q_{\kp,\sigma}$ be the cube having the same center as $Q_{\kp}$ but with side length $(1+\sigma){\kp}^{-1/2}$. Let $\chi\in C^{\infty}_{c}(Q_{\kp,\sigma})$ be a cut-off function satisfying, for all $\kappa\geq 1$,
\[
\chi=1 \quad{\rm in}\quad Q_{\kp}, \quad 0\leq \chi\leq 1 \quad {\rm and}\quad |\nb \chi|\leq C \sigma^{-1/2}{\kappa}^{1/2}\quad{    \rm in      }\quad Q_{\kp,\sigma}.
\]
An integration by parts and the first equation in \eqref{E-GL} yield the following localization formula
\begin{equation}\label{eq:bnd1}
\mathcal{E}_{0}(\chi\psi,{\bf A};Q_{\kappa,\sigma})= \kp^{2}\int_{Q_{\kp,\sigma}}\chi^{2}\Big(-1+\frac{1}{2}\chi^{2}\Big)|\psi|^{4}dx+ \int_{Q_{\kp,\sigma}} |\nb \chi|^{2}|\psi|^{2}dx
\leq C\sigma^{-1}\kappa^{-1/2}.
\end{equation}
Note  that we have used that the term $(-1+\frac{1}{2}\chi^{2})$ is negative, the bound on $|\nabla \chi|$ and that $|Q_{\kappa,\sigma}|\leq C {\kappa}^{-3/2}$.
Let us introduce the following  linear energy
\[
\mathcal{L}_{0,\kp}(\chi\psi,{\bf A})= \int_{Q_{\kp,\sigma}}\left(|(\nb -i\kp H{\bf A})\chi\psi|^{2}-\kappa^{2}|\chi\psi|^{2}\right)dx\,.
\]
Let $\phi$ be the function satisfying \eqref{gaugeAF} in $Q_{\kp,\sigma}$ (i.e. with $\ell=(1+\sigma)\kappa^{-1/2}$). Using the Cauchy-Schwarz inequality, we write,
\begin{equation}
\begin{aligned}
\mathcal{L}_{0,\kp}(\chi\psi,{\bf A})&= \mathcal{L}_{0,\kp}(e^{-i\kp H\phi}\chi\psi,{\bf A}-\nb \phi)\\
&\geq \int_{Q_{\kp,\sigma}}\left[(1-\kp^{-1/2})|(\nb-i\kp H{\bf F})e^{-i\kp H\phi} \chi\psi|^{2}-\kp^{2}|\chi \psi|^{2}-C{\kappa}^{1/2} H^2\lambda^{1/3} \ell^{2}|\chi \psi|^{2}\right]dx\,.
\end{aligned}
\end{equation}
Using the expression of $\lambda$ in Lemma~\ref{lem:gauge} and the assumption on $H$ in Theorem~\ref{thm:op-l4}, we get
\begin{multline}
\mathcal{L}_{0,\kp}(\chi\psi,{\bf A})= \mathcal{L}_{0,\kp}(e^{-i\kp H\phi}\chi\psi,{\bf A}-\nb \phi)\\
\geq \int_{Q_{\kp,\sigma}}\left[(1-\kp^{-1/2})|(\nb-i\kp H{\bf F})e^{-i\kp H\phi} \chi\psi|^{2}-\kp^{2}|\chi \psi|^{2}-C\kappa^{3/2}|\chi \psi|^{2}\right]dx\,.
\end{multline}

Let $b=(1-\kp^{-1/2})\frac{H}{\kp}$, and $R=\ell\sqrt{\kappa H}$ and $x_{\kp} $ the center of the square $Q_{\kp,\sigma}$. Apply the change of variables $y=\sqrt{\kp H}(x-x_{\kp})$ to get
\[
\mathcal{L}_{0,\kp}(\chi\psi,{\bf A})\geq {\kappa^{5/4}H^{-3/4}} {\mathcal M}_{0}(b,R) \norm{\chi \psi}^{2}_{4}-C{\kappa}^{3/2} \|\chi \psi\|_{2}^{2},
\]
where  $\mathcal{M}_{0}(b,R)$ is the energy introduced in \eqref{matcalM}. We use Theorem~\ref{NF} to write a lower bound of $\mathcal M_0(b,R)$ and  H\"older inequality to estimate $\|\chi\psi\|_2$. That way we get,
\[
\mathcal{L}_{0,\kp}(\chi\psi,{\bf A})\geq -{\kappa^{5/4}H^{-3/4}} R^{3/2}(-2g(b))^{1/2} \norm{\chi \psi}^{2}_{4}-C{\kappa}^{3/4}\|\chi \psi\|_{4}^{2}.
\]
Recall that $\ell=(1+\sigma)\kappa^{-1/2}$ is the side length of the cube $Q_{\kp,\sigma}$ and that $R=\ell \sqrt{\kp H}=(1+\sigma)\sqrt{H}$. Note that
\[
\begin{aligned}
\mathcal{E}_{0}(\chi\psi,{\bf A};Q_{\kappa,\sigma})&= \mathcal{L}_{0,\kp}(\chi\psi,{\bf A})+\dfrac{\kp^{2}}{2}\norm{\chi\psi}^{4}_{4}\\
&\geq -\kappa^{5/4}(1+\sigma)^{3/2}(-2g(b))^{1/2} \norm{\chi \psi}^{2}_{4}-C{\kappa}^{3/4}\|\chi \psi\|_{4}^{2}+\dfrac{\kp^{2}}{2}\norm{\chi\psi}^{4}_{4}\,.
\end{aligned}
\]
We insert this into \eqref{eq:bnd1} to get
%
\begin{equation}\label{Eq:bef-split}
\kp^{5/4}\left( -(1+\sigma)^{3/2}(-2g(b))^{1/2} -C{\kappa}^{-1/2} +\dfrac{\kappa^{3/4}}{2} \|\chi \psi\|_{4}^{2}\right)\|\chi \psi\|_{4}^{2}\leq C\sigma^{-1}\kp^{-1/2}.
\end{equation}
Two cases may occur~:

{\bf Case~I:}
\begin{equation*}
\left( -(1+\sigma)^{3/2}(-2g(b))^{1/2} -C{\kappa}^{-1/2} +\dfrac{\kappa^{3/4}}{2} \|\chi \psi\|_{4}^{2}\right)\leq \kp^{-1/2}
\end{equation*}

{\bf Case~II:}
\begin{equation*}
\left( -(1+\sigma)^{3/2}(-2g(b))^{1/2} -C{\kappa}^{-1/2} +\dfrac{\kappa^{3/4}}{2} \|\chi \psi\|_{4}^{2}\right)\geq \kp^{-1/2}.
\end{equation*}

In both cases, we infer from \eqref{Eq:bef-split},
\begin{equation}
\|\chi \psi\|_{4}^{2}\leq (1+\sigma)^{3/2}\kp^{-3/4} (-2g(b))^{1/2}+C\sigma^{-1}\kp^{-5/4}.
\end{equation}
Since $\chi=1$ in $Q_{\kp}\subset Q_{\kp,\sigma}$ and $|Q_\kappa|=\kappa^{-3/2}$, it follows that
 \begin{equation}\label{eq:bnd2}
\left(\frac1{|Q_\kappa|}\int_{Q_{\kp}}|\psi|^{4}dx\right)^{1/2}\leq (1+\sigma)^{3/2} (-2g(b))^{1/2}+C\sigma^{-1}\kp^{-1/2}.
\end{equation}
This yields the conclusion in Theorem~\ref{thm:op-l4} once we choose $\sigma=\left[\left(1-\frac{H}\kappa\right)\kappa^{1/2}\right]^{-1/2}$. In fact, Assumption~\ref{assump} ensures that
\begin{itemize}
\item $\sigma\ll1$ and $\sigma^{-1}\kappa^{-1/2}\ll 1-\frac{H}\kappa$\,;
\item $b=(1-\kappa^{-1/2})\frac{H}\kappa\to 1_-$ so that by Theorem~\ref{thm:Ab}, 
$g(b)=E_{\rm Ab}(b-1)^2+(b-1)^2o(1)$.
\end{itemize}
%
\end{proof}

We will need to work with  boxes rather than cubes only. These boxes are defined in: 

\begin{definition}\label{defn:box}
Let $0<\ell,L<1$. By a $(\ell,L)$ box we mean a cuboid of the form 
$$Q_{\ell,L}=(-\ell/2,\ell/2)\times(-\ell/2,\ell/2)\times (-L/2,L/2)+x_0\,,$$
for some point $x_0\in\R^3$ (the center of the box).
\end{definition}

Note that, a $(\ell,L)$ box for which $L=\ell$ is simply a cube of side length $\ell$.

\begin{rem}\label{rem:op-l4}
As a simple corollary of Theorem~\ref{thm:op-l4}, there exist  two constants $C>0$ and $\kappa_0>0$ such that the following estimate
$$\int_{Q_{\ell,L}}|\psi|^4\,dx\leq C\ell^2L\left(1-\frac{H}{\kappa}\right)^2\,,$$
is valid as long as Assumption~\ref{assump} is satisfied and
\begin{itemize}
\item $\kappa\geq \kappa_0$\,;
\item $ \kappa^{-1/2}\leq \ell,L<1$\,;
\item $Q_{\ell,L}\subset\{{\rm dist}(x,\partial\Omega)\geq 2\kappa^{-1/2}\}$ is a $(\ell,L)$-box.
\end{itemize}
Furthermore, it holds,
\begin{equation}\label{eq:op-l4*}
\limsup_{\kappa\to\infty}\left(\left(1-\frac{H}{\kappa}\right)^{-2}\frac{1}{|Q_{\ell,L}|}\int_{Q_{\ell,L}}|\psi|^4\,dx\right)\leq -2 E_{\rm Ab}\,.
\end{equation}
\end{rem}

\begin{corollary}\label{corol:op-l4}
Under the assumptions in Theorem~\ref{thm:op-l4}, 
\begin{equation}\label{eq:op-l4**}
\limsup_{\kappa\to\infty}\left(\left(1-\frac{H}{\kappa}\right)^{-2}\frac{1}{|D|}\int_{D}|\psi|^4\,dx\right)\leq -2 E_{\rm Ab}\,,
\end{equation}
where $D\subset\Omega$ is an open subset such that $|\partial D|=0$.
\end{corollary}

\section{Energy asymptotics}\label{sec:en}

In the sequel, we will work with the local energy introduced in \eqref{E0}. Also, we will use the notation introduced below.

\begin{notation}\label{not:Q-l}
For every $\ell\in(0,1)$, we let $Q_\ell\subset\Omega$ be a cube of side length $\ell$ and $\chi_{\ell}\in C^{\infty}_{c}(Q_{\ell})$ be a cut-off function satisfying
\begin{equation}\label{eq:chi-l}
  \chi_{\ell} =1\quad {\rm in}\quad  Q_{\ell-\frac{1}{\sqrt{\kp H}}},\quad 0\leq \chi_{\ell}\leq 1\,,\quad |\nb \chi_{\ell}|\leq c\sqrt{\kp H}\quad{\rm and}\quad |\Delta \chi_{\ell}|\leq c^2\kp H\quad {\rm in}\quad Q_{\ell},
\end{equation}
where $c>0$ is a universal constant.
\end{notation}

\begin{proposition}\label{prop-1}
There exist two constants $\kappa_0>1$ and $C>0$ such that the following inequalities holds
\begin{align*}
\dfrac{(1-\delta)}{|Q_\ell|}&\mathcal E_0(\chi_\ell\psi e^{i\kappa H\phi},\Fb;Q_\ell)\\
&\leq \dfrac{1}{|Q_{\ell}|} \mathcal{E}_{0}(\chi_\ell\psi,{\bf A};Q_{\ell})+ C\Big(\delta\kappa+\delta^{-1}\kappa^{1/3}\ell^2[\kappa-H]^{2/3}\Big)[\kp-H]\\
&\leq 
\dfrac{1}{|Q_{\ell}|} \mathcal{E}_{0}(\psi,{\bf A};Q_{\ell})
+ C\Big(\ell^{-1/2}\kappa^{1/2}+\delta\kappa+\delta^{-1}\kappa^{1/3}\ell^2[\kappa-H]^{2/3}\Big)[\kp-H]\,,
\end{align*}
where
\begin{itemize}
\item $\delta\in(0,1)$, $\kappa\geq\kappa_0$, and $(\kappa,H)$ satisfy Assumption~\ref{assump}\,;
\item $(\psi, {\bf A})\in H^{1}(\Omega ;\mathbb{C}) \times \dot{H}^{1}_{\ddiv,{\bf F}}(\R^{3})$ is a solution  of  \eqref{E-GL}\,;
\item $\kappa^{-1/2}\leq\ell<1$, $Q_\ell$ and $\chi_\ell$ are as in Notation~\ref{not:Q-l}\,;
\item $\Fb$ is the magnetic potential introduced in \eqref{MP-F}\,;
\item $\phi\in C^\infty(\overline{Q_\ell})$ is the smooth function in Lemma~\ref{lem:gauge}\,.
\end{itemize} 
\end{proposition}
\begin{proof}~

{\bf Step~1: Lower bound on $\mathcal{E}_{0}(\psi,\Ab;Q_{\ell})$.}
The aim of this step is to prove the estimate in \eqref{eq:decomp2} below.
Since $\chi_\ell=1$ in $Q_{\ell-\frac1{\sqrt{\kappa H}}}$, it holds the simple decomposition
\begin{equation}\label{eq:decomp1}
\mathcal E_0(\chi_\ell\psi,\Ab;Q_\ell)=\mathcal E_0(\psi,\Ab;Q_{\ell-\frac1{\sqrt{\kappa H}}})+\mathcal E_0(\chi_\ell\psi,\Ab;Q_\ell\setminus Q_{\ell-\frac1{\sqrt{\kappa H}}})\,.
\end{equation}
Straight forward calculations yield
\begin{align*}
\int_{Q_\ell\setminus Q_{\ell-\frac1{\sqrt{\kappa H}}}}&|(\nabla-i\kappa H\Ab)\chi_\ell\psi|^2\,dx\\
&=\int_{Q_\ell\setminus Q_{\ell-\frac1{\sqrt{\kappa H}}}}|\chi_\ell(\nabla-i\kappa H\Ab)\psi|^2\,dx+\int_{Q_\ell\setminus Q_{\ell-\frac1{\sqrt{\kappa H}}}}|\nabla\chi_\ell|^2|\psi|^2\,dx\\
&\quad+2{\rm Re}\left\{\int_{Q_\ell\setminus Q_{\ell-\frac1{\sqrt{\kappa H}}}}\chi_\ell\overline{\psi}\nabla\chi_\ell\cdot(\nabla-i\kappa H\Ab)\psi\,dx\right\}\\
&=\int_{Q_\ell\setminus Q_{\ell-\frac1{\sqrt{\kappa H}}}}|\chi_\ell(\nabla-i\kappa H\Ab)\psi|^2\,dx-\int_{Q_\ell\setminus Q_{\ell-\frac1{\sqrt{\kappa H}}}}|\psi|^2\chi_\ell\Delta\chi_\ell\,dx\,.
\end{align*}
We insert the estimates in  Remark~\ref{rem:op-l4} into the aforementioned formula to obtain
$$
\int_{Q_\ell\setminus Q_{\ell-\frac1{\sqrt{\kappa H}}}}|(\nabla-i\kappa H\Ab)\chi_\ell\psi|^2\,dx\leq
\int_{Q_\ell\setminus Q_{\ell-\frac1{\sqrt{\kappa H}}}}|(\nabla-i\kappa H\Ab)\psi|^2\,dx
+C\ell^{-1/2}\kappa^{1/2}[\kappa-H]\ell^3\,.
$$
We insert this into \eqref{eq:decomp2}. After a rearrangement of the terms we get
$$
\mathcal E_0(\chi_\ell\psi,\Ab;Q_\ell)\leq\mathcal E_0(\psi,\Ab;Q_\ell)+\kappa^2\int_{Q_\ell}(1-\chi_\ell^2)|\psi|^2\,dx
+C\ell^{-1/2}\kappa^{1/2}[\kappa-H]\ell^3\,.
$$
We estimate the term $\displaystyle\int_{Q_\ell}(1-\chi_\ell^2)|\psi|^2\,dx$ using the assumption on the support of $1-\chi_\ell$, the Cauchy-Schwarz inequality and the estimate in Remark~\ref{rem:op-l4}. That way we get
\begin{equation}\label{eq:decomp2}
\mathcal E_0(\chi_\ell\psi,\Ab;Q_\ell)\leq\mathcal E_0(\psi,\Ab;Q_\ell)+C\ell^{-1/2}\kappa^{1/2}[\kappa-H]\ell^3\,.
\end{equation}

{\bf Step~2: Replacing $\Ab$ by $\Fb$.}

Let $\phi\in C^\infty(\overline{Q_\ell})$ be the function satisfying the estimate in \eqref{gaugeAF}. Using the gauge invariance and the Cauchy-Schwarz inequality, we get 
\begin{multline*}
\mathcal E_0(\chi_\ell\psi,\Ab;Q_\ell)=\mathcal E_0(\chi_\ell\psi e^{i\kappa H\phi},\Ab-\nabla\phi;Q_\ell)\\
\geq (1-\delta)\mathcal E_0(\chi_\ell\psi e^{i\kappa H\phi},\Fb;Q_\ell)
-\Big(C\delta^{-1}\kappa^2H^2\|\Ab-\Fb-\nabla\phi\|_{L^\infty(Q_\ell)}^2+\delta\kappa^2\Big)\int_{Q_\ell}|\psi|^2\,dx\,.
\end{multline*}
Using the estimates in Remark~\ref{rem:op-l4} and \eqref{gaugeAF} we get,
$$
\mathcal E_0(\chi_\ell\psi,\Ab;Q_\ell)\geq
(1-\delta)\mathcal E_0(\chi_\ell\psi e^{i\kappa H\phi},\Fb;Q_\ell)-C\Big(\delta^{-1}\kappa^2\left(1-\frac{H}\kappa\right)^{5/3}\ell^5+\delta\kappa[\kappa-H]\ell^3\Big)\,.
$$
Inserting this into \eqref{eq:decomp2}, we finish the proof of Proposition~\ref{prop-1}.
\end{proof}

\begin{rem}\label{rem:lb-en}
In the setting of Proposition~\ref{prop-1}, let $R=\ell\sqrt{\kappa H}$. The change of variables $x\mapsto x\sqrt{\kappa H}$, Lemma~\ref{Lem1} and \eqref{eq:m0gb} yield
$$\dfrac{1}{|Q_\ell|}\mathcal E_0(\chi_\ell\psi e^{i\kappa H\phi},\Fb;Q_\ell)\geq
\kappa^2g\left(\frac{H}\kappa\right)\,.
$$
Furthermore, under Assumption~\ref{assump}, we know that $H/\kappa\to1_-$, and by Theorem~\ref{thm:Ab}, 
$$\kappa^2g\left(\frac{H}\kappa\right)=E_{\rm Ab}[\kappa-H]^2+[\kappa-H]^2o(1)\,.$$
\end{rem}

\begin{proposition}\label{prop-2}
There exist positive constants $C>0$ and $\kp_{0}>1$ such that the following inequality holds
\begin{multline*}
\dfrac{\mathcal{E}_{0}(\psi,{\bf A};Q_{\ell})}{|Q_{\ell}|}\leq
 (1+\delta)\left(1-\frac{2}{R}\right)[\kp-H]_{+}^{2}\dfrac{c(R)}{R^{2}}\\
 +C\Big(\ell^{-1}+\kappa^{-1}\ell^{-3}[\kappa-H]^{-1}+\delta\kappa+\delta^{-1}\kappa^{1/3}\ell^2[\kappa-H]^{2/3}+\ell^{-1/2}\kappa^{1/2} \Big) [\kappa-H]\,,\end{multline*}
where
\begin{itemize}
\item $\delta\in(0,1)$, $\kappa\geq\kappa_0$, and $(\kappa,H)$ satisfy Assumption~\ref{assump}\,;
\item $(\psi, {\bf A})\in H^{1}(\Omega ;\mathbb{C}) \times \dot{H}^{1}_{\ddiv,{\bf F}}(\R^{3})$ is a minimizer of the functional in  \eqref{GL-Energy}\,;
\item $\kappa^{-1/2}\leq \ell<1$, $Q_\ell\subset\{{\rm dist}(x,\partial\Omega)\geq 2\kappa^{-1/2}\}$ is a cube of side length $\ell$\,;
\item $R=\ell\sqrt{\kappa H}$ and $c(R)$ is the energy introduced in \eqref{cR}.
\end{itemize}
\end{proposition}
\begin{proof}
Let $x_{0}$ be the center of $Q_{\ell}$. Without loss of generality, we may assume that $x_{0}=0$ so that we reduce to the case
$$
Q_{\ell}= (-\ell/2,\ell/2)\times(-\ell/2,\ell/2)\times(-\ell/2,\ell/2)\subset \{{\rm dist}(x,\partial\Omega)>\kp^{-1+\delta}\}.
$$
In light of Lemma~\ref{lem:gauge}, we may assume, after performing a gauge transformation,  that the magnetic potential satisfies,
\begin{equation}\label{gaugeAF*}
|{\bf A}(x)-{\bf F}(x)|\leq C\kp^{-1}\left(1-\frac{H}\kappa\right)^{1/3}\ell,
\end{equation}
where ${\bf F} $ is the magnetic potential introduced in \eqref{MP-F}.

Let $b=H/\kp$, $R=\ell\sqrt{\kp H}$ and $v_{R}\in H^{1}_{0}(Q_{R})$ be a minimizer of the functional in \eqref{M0},  i.e. $F_{b,Q_{R}}(v_{R})= M_{0}(b,R)$.

Let $\chi_R \in C^{\infty}_{c}(\R^{3})$ be a cut-off function such that
\begin{equation}
0\leq \chi_{R}\leq 1,~|\nabla \chi_{R}|\leq C  \quad {\rm in }\quad {\rm supp }~\chi_{R}\subset Q_{R+1}, \quad \chi_{R} =1\quad {\rm in}~ Q_{R},
\end{equation}
for some universal constant $C$. Let $\eta_{R}(x)= 1-\chi_{R}(x\sqrt{\kp H})$ for all $x\in\R^{3}$. We introduce the function (cf. \cite{SS02})
\begin{equation}
\varphi(x)={\bf 1}_{Q_{\ell}}(x) v_{R}(x\sqrt{\kp H})+\eta_{R}(x)\psi(x), \quad (x\in\Omega).
\end{equation}
Note that the function $\varphi$ satisfies
\begin{equation}
\varphi(x)=\left\{ \begin{array} {lcl}
v_{R}(x\sqrt{\kp H})& {\rm if}& x\in Q_{\ell}\,,\\
\eta_{R}(x)\psi(x)&{\rm if}& x\in Q_{\ell+\frac{1}{\sqrt{\kp H}}}\setminus Q_{\ell}\,,\\
\psi(x)&{\rm if}& x\in \Omega \setminus Q_{\ell+\frac{1}{\sqrt{\kp H}}}\,.\\
\end{array}\right.
\end{equation}
We will prove that, for all $\delta\in(0,1)$,
 \begin{equation}\label{E0-UB-est}
 \mathcal{E}(\varphi,{\bf A};\Omega)\leq  \mathcal{E}(\psi,{\bf A};\Omega\setminus Q_{\ell})+(1+\delta)\dfrac{1}{b\sqrt{\kp H}}M_{0}(b,R)+r(\kp)
  \end{equation}
where $M_{0}(b,R)$ is defined in \eqref{M0},  $r(\kp)$ is 
\begin{equation}\label{eq:r*}
r(k)= C\Big(\delta\kappa+\delta^{-1}\kappa^{1/3}\ell^2[\kappa-H]^{2/3}+\ell^{-1/2}\kappa^{1/2} \Big)[\kappa-H]\ell^3\,,
\end{equation}
and $C>0$ is a constant.
\begin{proof}[Proof of \eqref{E0-UB-est}]
Recall the Ginzburg-Landau energy $\mathcal{E}_{0}$ defined in \eqref{E0}. We may write
\begin{equation}\label{eq:E1+E2}
 \mathcal{E}(\varphi,{\bf A};\Omega)=\mathcal{E}_{1}+\mathcal{E}_{2}
 \end{equation}
where
\begin{equation}\label{eq:E1,E2}
\mathcal{E}_{1}= \mathcal{E}(\varphi,{\bf A};\Omega\setminus Q_{\ell}),\qquad \mathcal{E}_{2}=\mathcal{E}_0(\varphi,{\bf A};Q_{\ell})
\end{equation}
Let us start by estimating $\mathcal{E}_{1}$ from above. We write
\begin{equation}\label{eq:E1}
\mathcal{E}_{1}= \mathcal{E}(\psi,{\bf A};\Omega\setminus Q_{\ell})+ \mathcal{R}(\psi,{\bf A}),
\end{equation}
where
$$\mathcal{R}(\psi,{\bf A})=\mathcal E_0\big(\eta_R(x\sqrt{\kappa H})\psi,\Ab;Q_{\ell+\frac{1}{\sqrt{\kp H}}}\setminus Q_{\ell}\big)-
\mathcal E_0\big(\psi,\Ab;Q_{\ell+\frac{1}{\sqrt{\kp H}}}\setminus Q_{\ell}\big)\,.$$
An integration by parts yields
\begin{multline*}
 \mathcal{R}(\psi,{\bf A})=\frac{\kappa^2}2\int_{  Q_{\ell+\frac{1}{\sqrt{\kp H}}}\setminus Q_{\ell}} \big(\eta_{R}^{4}(x\sqrt{\kp H})-2\eta_{R}^{2}(x\sqrt{\kp H})-1\big)|\psi|^{4}dx
 +\kp^{2}\int_{  Q_{\ell+\frac{1}{\sqrt{\kp H}}}\setminus Q_{\ell}} |\psi|^{2}dx \\
-\int_{  Q_{\ell+\frac{1}{\sqrt{\kp H}}}\setminus Q_{\ell}}|(\nb -i\kp H{\bf A})\psi|^{2} dx +\int_{  Q_{\ell+\frac{1}{\sqrt{\kp H}}}\setminus Q_{\ell}}|\nb \eta_{R}|^{2}|\psi|^{2}dx.
 \end{multline*}
Using that $0\leq \eta_{R}\leq1$ together with the estimate $|\nb \eta_{R}|\leq C\sqrt{\kappa H}$ and Remark~\ref{rem:op-l4}, we  get
$$
 \mathcal{R}(\psi,{\bf A})\leq C\ell^{-1/2}\kappa^{1/2}[\kappa-H]\ell^3.
$$
By inserting this into \eqref{eq:E1}, we deduce that
\begin{equation}\label{eq:E1*}
\mathcal{E}_{1}\leq  \mathcal{E}(\psi,{\bf A};\Omega\setminus Q_{\ell})+C\ell^{-1/2}\kappa^{1/2}[\kappa-H]\ell^3.
\end{equation}
Now, we estimate the energy $\mathcal{E}_{2}$ in \eqref{eq:E1,E2}. Using the Cauchy-Shwarz inequality and \eqref{gaugeAF*}, we write for all $\delta\in(0,1)$,
\begin{multline*}
\mathcal{E}_{2}\leq (1+\delta)\int_{Q_{\ell}}\Bigg\{|(\nb -i\kp H{\bf F})\varphi|^{2}-\kp^{2}|\varphi|^{2}+\dfrac{\kp^{2}}{2}|\varphi|^{4}\Bigg\}dx\\
+C\Big(\delta\kappa^2+\delta^{-1}\kp^{2}\left(1-\frac{H}\kappa\right)^{2/3}\ell^2\Big)
\int_{Q_{\ell}}|\varphi|^{2}dx\,.
\end{multline*}
Now we use that $\varphi=v_{R}(x\sqrt{\kp }H)$ in $Q_{\ell}$, the estimate in Lemma~\ref{Lem:L4} and \eqref{bnd-on-g} to write,
\begin{multline}\label{eq:E2}
\mathcal{E}_{2}\leq (1+\delta)\int_{Q_{\ell}}\Bigg\{|(\nb -i\kp H{\bf F})\varphi|^{2}-\kp^{2}|\varphi|^{2}+\dfrac{\kp^{2}}{2}|\varphi|^{4}\Bigg\}dx\\
+C\Big(\delta\kappa[\kappa-H]\ell^3+\delta^{-1}\kp^{2}\left(1-\frac{H}\kappa\right)^{5/3}\ell^5\Big)\,.
\end{multline}
Since $\varphi(x)=v_{R}(x\sqrt{\kp H})$ in $Q_{\ell}$, $b=H/\kappa$ and $R=\ell \sqrt{\kp H}$, a change of variables yields
$$
\int_{Q_{\ell}}\Bigg\{|(\nb -i\kp H{\bf F})\varphi|^{2}-\kp^{2}|\varphi|^{2}+\dfrac{\kp^{2}}{2}|\varphi|^{4}\Bigg\}dx
=\dfrac{1}{b\sqrt{\kp H}} M_0(b,R).
$$
Inserting this into \eqref{eq:E2}  then collecting \eqref{eq:E1*} and \eqref{eq:E1+E2}, we finish the proof of \eqref{E0-UB-est}.
\end{proof}

Now we proceed in the proof of Proposition~\ref{prop-2}. By the definition of the minimizer $(\psi,\Ab)$, we have
\[
\mathcal{E}(\psi,{\bf A};\Omega)\leq \mathcal{E}(\varphi,{\bf A};\Omega).
\]
Since $\mathcal{E}(\psi,{\bf A};\Omega)=\mathcal{E}(\psi,{\bf A};\Omega\setminus Q_{\ell})+\mathcal{E}_{0}(\psi,{\bf A};Q_{\ell})$, then \eqref{E0-UB-est} yields,
$$
\mathcal{E}_{0}(\psi,{\bf A};Q_{\ell})\leq  (1+\delta) \dfrac{1}{b\sqrt{\kp H}}M_{0}(b,R)+
r(\kappa)\,,
$$
where $r(\kappa)$ is given in \eqref{eq:r*}. 
Dividing both sides by $|Q_{\ell}|$ and using Lemma~\ref{Lem1} and \eqref{eq:Ka}, we finish the proof of Proposition~\ref{prop-2}.
\end{proof}

\begin{rem}\label{rem:par}{\bf [Choice of the parameters]}
Let $\mu=\kappa^{1/2}(1-\frac{H}\kappa)$. Under Assumption~\ref{assump}, $1\ll\mu\ll\kappa^{1/2}$.

Let $B>0$ be a function of $\kappa$ such that $1\ll B\ll \mu$. We choose $\delta=B\kappa^{-1/2}$. Under the additional condition $\mu^{-2}\ll\ell\ll 1$, we observe that all the terms
$$\delta\kappa,\quad \ell^{-1/2}\kappa^{1/2}\,,\quad \ell^{-1}\,,\quad \kappa^{-1}\ell^{-3}[\kappa-H]^{-1}$$
are of the order $o([\kappa-H])$. 

To get $\delta^{-1}\kappa^{1/3}[\kappa-H]^{2/3}=o([\kappa-H])$, 
the additional condition $\ell\approx\mu^{1/6}\kappa^{-1/3}$ arises. To respect the condition $\ell\gg\mu^{-2}$, $\mu$ should satisfy  $\mu\gg\kappa^{2/13}$. This motivates  Assumption~\ref{assump'} below. 
\end{rem}

\begin{assumption}\label{assump'}~
\begin{itemize}
\item $a:\R_+\to\R_+$ and $b:\R_+\to\R_+$ are two functions satisfying
$$\lim_{\kappa\to\infty}a(\kappa)=\infty\,,\quad\lim_{\kappa\to\infty}b(\kappa)=0\quad{\rm and}\quad a(\kappa)\kappa^{-9/26}\leq b(\kappa){~\rm in~a~neighborhood~of~}\infty\,.$$
\item $\kappa>0$ and $H>0$ satisfy $\displaystyle a(\kappa)\kappa^{-9/26}\leq 1-\frac{H}\kappa\leq b(\kappa)$.
\end{itemize}
\end{assumption}

Collecting Propositions~\ref{prop-1} and \ref{prop-2}, we get:

\begin{corollary}\label{corol:lb-op-l4}
There exist $\kappa_0>0$ and a function ${\rm err}:[\kappa_0,\infty)\to(0,\infty)$  such that:
\begin{itemize}
\item $\displaystyle\lim_{\kappa\to\infty}{\rm err}(\kappa)=0$\,;
\item the following two inequalities  hold
\begin{equation}\label{eq:lb-en-l4}
\left|\frac1{|Q_\ell|}\mathcal E_0(\chi_\ell\psi e^{i\kappa H\phi},\Fb;Q_\ell)-
[\kappa-H]^2E_{\rm Ab}\right|\leq [\kappa-H]^2{\rm err}(\kappa)\,,
\end{equation}
\begin{equation}\label{eq:lb-op-l4}
\left|\dfrac1{|Q_{\ell}|}\int_{Q_{\ell}}{|\psi|^{4}}dx+2E_{\rm Ab}\Big(1-\dfrac{H}{\kp}\Big)^{2}\right|\leq \Big(1-\dfrac{H}{\kp}\Big)^{2}{\rm err}(\kappa)\,,
\end{equation}
where 
\begin{itemize}
\item $E_{\rm Ab}$ is the Abrikosov constant introduced in Theorem~\ref{thm:Ab}\,;
\item $\Fb$ is the magnetic potential in \eqref{MP-F}\,;
\item $\kappa\geq\kappa_0$ and $(\kappa,H)$ satisfy Assumption~\ref{assump'}\,;
\item $(\psi,\Ab)$ is a minimizer of  \eqref{GL-Energy}\,;
\item $\ell=(\kappa H)^{-1/2}\sqrt{2\pi[(\kappa-H)^{1/3}\kappa^{1/6}H]}$ with $[\cdot]$ denoting the integer part (floor function)\,;
\item $Q_{\ell}\subset \{{\rm dist}(x,\partial\Omega)>2\kp^{-1/2}\}$ and $\chi_\ell$ are as in Notation~\ref{not:Q-l}\,;
\item $\phi\in C^\infty(\overline{Q_\ell})$ is the function defined by Lemma~\ref{lem:gauge}.
\end{itemize}
\end{itemize}
\end{corollary}
\begin{proof}
Under Assumption~\ref{assump'}, we know that $\kappa^{-9/26}\ll1-\frac{H}\kappa\ll\kappa^{-1/2}$.
We choose $\delta=B\kappa^{-1/2}$ where $B>0$ is a function of $\kappa$ satisfying $1\ll B\ll\mu:=\kappa^{1/2}(1-\frac{H}\kappa)$. Note that our choice of $\ell$ verifies
$\ell\approx\mu^{1/6}\kappa^{-1/3}$. As explained in Remark~\ref{rem:par}, with this choice, we get that all the remainder terms in Proposition~\ref{prop-1} and \ref{prop-2} are of order $o([\kappa-H]^2)$.

Now, collecting the estimates in Proposition~\ref{prop-1}, \ref{prop-2} and 
Remark~\ref{rem:lb-en}, we get
\begin{multline*}
(1-\delta)\kappa^2g\left(\frac{H}\kappa\right)\leq \frac{(1-\delta)}{|Q_\ell|}\mathcal E_0(\chi_\ell\psi e^{i\kappa H\phi},\Fb;Q_\ell)\\
\leq \dfrac{\mathcal E_0(\chi_\ell\psi ,\Ab;Q_\ell)}{|Q_{\ell}|}+o([\kappa-H]^2)\leq \frac{c(R)}{R^2}[\kappa-H]^2+o([\kappa-H]^2)\,,
\end{multline*}
where $R=\ell\sqrt{\kappa H}$. Our choice of $\ell$ ensures that $R\gg 1$ and $(2\pi)^{-1}R^2\in\mathbb N$. By applying  \eqref{eq:m0gb} and Theorem~\ref{thm:Ab},
we get \eqref{eq:lb-en-l4} and
\begin{equation}\label{eq:lb-en-l4*}
\mathcal E_0(\chi_\ell\psi ,\Ab;Q_\ell)\leq \ell^{3} [\kappa-H]^2E_{\rm Ab}+\ell^{3} o([\kappa-H]^2)\,.
\end{equation}
The proof of \eqref{eq:lb-op-l4} follows from the following localization formula,
$$\mathcal E_0(\chi_\ell\psi ,\Ab;Q_\ell)
= \kp^{2}\int_{Q_{\ell}}\chi_\ell^{2}\Big(-1+\frac{1}{2}\chi_\ell^{2}\Big)|\psi|^{4}dx+ \int_{Q_{\ell}} |\nb \chi_\ell|^{2}|\psi|^{2}dx\,.
$$
By inserting \eqref{eq:lb-en-l4*} into the aforementioned formula and by using that $\chi_\ell=1$ in $Q_{\ell-\frac1{\sqrt{\kappa H}}}$, we get
$$ \frac{-\kp^{2}}2\int_{Q_{\ell-\frac1{\sqrt{\kappa H}}}}|\psi|^{4}dx\leq 
[\kappa-H]^2E_{\rm Ab}\ell^3+
\kappa^2\int_{Q_{\ell}\setminus Q_{\ell-\frac1{\sqrt{\kappa H}}}}|\psi|^{4}dx- \int_{Q_{\ell}} |\nb \chi_\ell|^{2}|\psi|^{2}dx+\ell^3o([\kappa-H]^2)\,.
$$
The estimate in Remark~\ref{rem:op-l4} yields that
$$\kp^{2}\int_{Q_{\ell}\setminus Q_{\ell-\frac1{\sqrt{\kappa H}}}}|\psi|^{4}dx+\int_{Q_{\ell}} |\nb \chi_\ell|^{2}|\psi|^{2}dx\leq 
C\ell^{-1/2}\kappa^{1/2}[\kappa-H]\ell^3=\ell^3o([\kappa-H]^2)\,.
$$
This and Theorem~\ref{thm:op-l4} (also see Remark~\ref{rem:op-l4}) finish the proof of \eqref{eq:lb-op-l4}.
\end{proof}

\section{Sharp estimate of the $L^2$-norm}\label{sec:l2}

This section contains three main results:
\begin{itemize}
\item Lemma~\ref{lem:op-3D} regarding the spectral theory of the Landau Hamiltonian with (magnetic) periodic conditions with respect to a box lattice of $\R^3$\,;
\item Lemma~\ref{thm:Pi-Ab} and Theorem~\ref{corol:Pi-Ab} regarding the behavior of the  minimizers of the functional in \eqref{GL-Energy} in cubes with small lengths.
\end{itemize}

The proof of Theorem~\ref{thm:op-l2-main} is  a simple consequence of  the result summarized in Theorem~\ref{corol:Pi-Ab}. The proof of Theorem~\ref{corol:Pi-Ab} relies on Lemma~\ref{thm:Pi-Ab}. The proof of Lemma~\ref{thm:Pi-Ab} needs the result in Lemma~\ref{lem:op-3D} as a key ingredient. 


\subsection{The $3D$ periodic operator}

Let  $R>0$ such that $R^2\in2\pi\mathbb N$, $L>0$   and $\Fb$ be the magnetic potential in \eqref{MP-F}. 
We denote by $P_{R,L}^{3D}$ the operator
\[
P_{R,L}^{3D}= -(\nb -i{\bf F})^{2} \quad {\rm in}\quad  L^{2}_{\rm per}(Q_{R,L}),\quad Q_{R,L}=(-R/2,R/2)^{2}\times(-L/2,L/2)\,,
\]
with form domain the space ${E}_{R}^{3D}$ defined as follows
\begin{equation}\label{eq:3D-Fdom}
\begin{aligned}
{E}_{R,L}^{3D}= \Big\{  u\in H^{1}_{\rm loc}(\R^{3};\mathbb{C})~:~&          u(x_{1}+R,x_{2},x_{3})=e^{-i Rx_{2}/2} u(x_{1},x_{2},x_{3})\,,\\
& u(x_{1},x_{2}+R,x_{3})=e^{i Rx_{1}/2}u(x_{1},x_{2},x_{3})\,,\\
&u(x_{1},x_{2},x_{3}+L)= u(x_{1},x_{2},x_{3})\,, \quad \forall~ (x_{1},x_{2},x_{3})      \in \R^{3}                   \Big\}\,.
\end{aligned}
\end{equation}
When $L=R$, we will omit the reference to $L$ in the notation and simply write
$P^{3D}_{R}$, $E_R^{3D}$ and $Q_R$.

The operator $P_{R,L}^{3D}$ is  with compact resolvent. Its sequence of increasing {\it distinct} eigenvalues  is  denoted by $\{\mu_{j}({P}_{R,L}^{3D})\}$.

The Fourier transform with respect to the $x_3$-variable allows us to separate variables and express the operator $P^{3D}_{R,L}$ as the direct sum
\begin{equation}\label{eq:3D-2D}
\bigoplus_{n\in\mathbb Z}\Bigg(P_R^{2D}+(2\pi n L^{-1})^2\Bigg)\quad{\rm in}~\bigoplus_{n\in\mathbb Z}L^2\big((-R/2,R/2)^2\big)\,,
\end{equation}
where $P_R^{2D}$ is the operator introduced in \eqref{eq:op-2D}.
Consequently, we get 
\begin{equation}\label{eq:3D-sp}
\mu_{1}({P}_{R,L}^{3D})=1\quad{\rm and}\quad \mu_{2}({P}_{R,L}^{3D})=1+4\pi^2L^{-2}\,.
\end{equation}
Let $\Pi_1$ be the orthogonal projection on $L_R\subset L^2((-R/2,R/2)^2)$, the first eigenspace of the operator $P_R^{2D}$ in \eqref{eq:op-2D}. By Proposition~\ref{prop:op-2D}, we know that, under the assumption that $R^2\in 2\pi\mathbb N$, the space $L_R$ is finite dimensional and the dimension is equal to $N:=R^2/2\pi$. Thus, we may express the orthogonal projection $\Pi_1$ as follows,
$$\forall~g\in L^2((-R/2,R/2)^2)\,,\quad \Pi_1u=\sum_{m=1}^N\langle g,f_m\rangle_{L^2((-R/2,R/2)^2)}f_m\,,$$
where $(f_m)$ is an orthonormal basis of the space $L_R$. That way, we may view $\Pi_1$ as 
a projection in the space $L^2(Q_{R,L})$ via the formula
\begin{multline}\label{eq:Pi-3D}
\forall~u\in  L^2(Q_{R,L})\,,\\ (\Pi_1u)(x_1,x_2,x_3)=\sum_{m=1}^Nf_m(x_1,x_2)\int_{K_R} u(x_1,x_2,x_3)\overline{f_m(x_1,x_2)}\,dx_1dx_2\,,
\end{multline}
where
\begin{equation}\label{eq:KR}
K_R=(-R/2,R/2)\times(-R/2,R/2)\,.
\end{equation}
We introduce the quadratic form of the operator $P_R^{3D}$,
\begin{equation}\label{eq:qf-3D}
\mathcal Q_{R,L}^{3D}(u)=\int_{Q_{R,L}}|(\nabla-i\Fb)u|^2\,dx\,.
\end{equation}
Note that by definition of $\Fb$ and $\Ab_0$ in \eqref{MP-F} and \eqref{A0} respectively, we observe the following useful inequality,
\begin{equation}\label{eq:qf-3D*}
\mathcal Q_R^{3D}(u)=\int_{Q_{R,L}}\big(|(\nabla_{(x_1,x_2)}-i\Ab_0)u|^2+|\partial_{x_3}u|^2\big)\,dx\geq \int_{Q_{R,L}}|(\nabla_{(x_1,x_2)}-i\Ab_0)u|^2\,dx\,,
\end{equation}
where $\nabla_{(x_1,x_2)}=(\partial_{x_1},\partial_{x_2})$.

Now, we can prove the $3D$ analogue of Lemma~\ref{lem:op-2D}:

\begin{lem}\label{lem:op-3D}
Let $2\leq p\leq 6$.
There exists a constant $C_{p}>0$ such that for any $\gamma \in (0,1/2)$, $R,L>1$ and $u\in E_{R,L}^{3D}$ satisfying
\begin{equation}\label{Asp-gam}
\mathcal Q^{\rm 3D}_{R,L}(u)-(1+\gamma) \norm{u}^2_{L^{2}(Q_{R,L})}\leq 0
\end{equation}
then the following estimate holds:
$$
\norm{u-\Pi_{1}u}_{L^{p}(Q_{R,L})}\leq C_{p}\sqrt{\gamma}\norm{u}_{L^{2}(Q_{R,L})}.
$$
\end{lem}
\begin{proof}
Let $\Pi_2u=u-\Pi_1u$. It is easy to check that $\Pi_1u$ and $\Pi_2u$ are orthogonal in $L^2(Q_{R,L})$ and that
$$\mathcal Q^{\rm 3D}_{R,L}(u)-\norm{u}^2_{L^{2}(Q_{R,L})}=\sum_{i=1}^2\left(\mathcal Q^{\rm 3D}_{R,L}(\Pi_iu)-\norm{\Pi_iu}^2_{L^{2}(Q_{R,L})}\right)\,.$$
Using \eqref{eq:qf-3D*} and \eqref{eq:gap}, we get
$$\mathcal Q^{\rm 3D}_{R,L}(u)-\norm{u}^2_{L^{2}(Q_{R,L})}\geq \frac12\mathcal Q^{\rm 3D}_{R,L}(\Pi_2u)+\left(\frac32-1\right)\norm{\Pi_2u}^2_{L^{2}(Q_{R,L})}\,.$$
Using the diamagnetic inequality, we get further
$$\mathcal Q^{\rm 3D}_{R,L}(u)-\norm{u}^2_{L^{2}(Q_{R,L})}\geq \frac12\norm{\nabla|\Pi_2u|}^2_{L^2(Q_{R,L})}+\frac12\norm{\Pi_2u}^2_{L^{2}(Q_{R,L})}\,.$$
We insert this into \eqref{Asp-gam} to get,
$$\norm{\nabla|\Pi_2u|}^2_{L^2(Q_{R,L})}+\norm{\Pi_2u}^2_{L^{2}(Q_{R,L})}\leq 2\gamma
\norm{u}^2_{L^{2}(Q_{R,L})}\,.$$
This finishes the proof of Lemma~\ref{lem:op-3D} once the following Sobolev inequality is established
\begin{equation}\label{eq:op-3D-Sob}
\forall~R\geq 1\,,~\forall~p\in[2,6]\,,~\forall~f\in~E_R^{3D}\,,\quad \|f\|_{L^p(Q_{R,L})}\leq C_p\|f\|_{H^1(Q_{R,L})}\,,
\end{equation}
where $C_p$ is a constant independent from $R\geq 1$.
To prove \eqref{eq:op-3D-Sob}, let $f\in E_{R,L}^{3D}$, $\chi\in C_c^\infty(B_{\R^2}(0,6))$ and $\eta\in C_c^\infty(B_{\R}(0,6))$ such that 
\begin{itemize}
\item $\chi=1$ in $B_{\R^2}(0,3)$ and $\eta=1$ in $B_\R(0,3)$\,;
\item $0\leq \chi\leq 1$ in $B_{\R^2}(0,6)$ and $0\leq \eta\leq 1$ in $B_\R(0,3)$\,;
\end{itemize} Note that, since $f\in E_{R,L}^{3D}$, then $f(x)$ can be defined everywhere by (magnetic) periodicity. Let us define  
$$g(x)=\chi\left(\frac{x^\bot}{R}\right)\eta\left(\frac{x_3}{L}\right)f(x)\,,\quad (x=(x_\bot,x_3)\in\R^3)\,.$$
Clearly, $g$ belongs to the Homogeneous Sobolev space and the following Sobolev inequality holds
$$\|g\|_{L^6(\R^3)}\leq C\|\nabla g\|_{L^2(\R^3)}\,.$$
This yields \eqref{eq:op-3D-Sob} for $p=6$. By H\"older's inequality, we get \eqref{eq:op-3D-Sob} for all $2\leq p\leq 6$.
\end{proof}

\subsection{Average asymptotics}

Here we return back to the analysis of the minimizers of the functional in \eqref{GL-Energy}.

\begin{lemma}\label{thm:Pi-Ab}
There exist $\kappa_0>1$, $C>0$ and a function ${\rm err}:[\kappa_0,\infty)\to(0,\infty)$  such that it  holds the following
 \begin{equation}\label{eq:op-l2*}
 \norm{v-\Pi_{1}v}_{{L^{2}}((-R/2,R/2)^{3})}\leq C\sqrt{1-\frac{H}{\kp}}\norm{v}_{L^{2}((-R/2,R/2)^{3})}\,,
 \end{equation}
\begin{equation}\label{eq:op-l2**}
\mathcal{E}_{0}\big(e^{i\kappa H\phi}\chi_{\ell}\psi,{\bf F };Q_{\ell}\big)
\geq \dfrac{1}{\sqrt{\kp H}}\int_{(-R/2,R/2)^3}\left(\left(1-\frac{\kp}{H}\right)|\Pi_{1}v|^{2}+\dfrac{\kp}{2H}|v|^{4}\right) dx\,,\\
\end{equation}
\begin{equation}\label{eq:op-l4-truncated}
\frac1{R^3}\int_{(-R/2,R/2)^3}|v|^4\,dx= -2E_{\rm Ab}\left(1-\frac{H}\kappa\right)^2+\left(1-\frac{H}\kappa\right)^2\,{\rm err}(\kappa)\,,
\end{equation}
and
\begin{equation}\label{eq:op-l2-lb**}
\frac1{R^3}\int_{(-R/2,R/2)^3}|v|^2\,dx\geq -2E_{\rm Ab}\left(1-\frac{H}\kappa\right)+\left(1-\frac{H}\kappa\right)\,{\rm err}(\kappa)\,,
\end{equation}
where
\begin{itemize}
\item $\displaystyle\lim_{\kappa\to\infty}{\rm err}(\kappa)=0$\,;
\item $\Fb$ is the magnetic potential in \eqref{MP-F}\,;
\item $\kappa\geq\kappa_0$ and $(\kappa,H)$ satisfy Assumption~\ref{assump'}\,;
\item $(\psi,\Ab)$ is a minimizer of  \eqref{GL-Energy}\,;
\item $\ell=(\kappa H)^{-1/2}\sqrt{2\pi[(\kappa-H)^{1/3}\kappa^{1/6}H]}$ with $[\cdot]$ denoting the integer part (floor function)\,;
\item $R=\ell\sqrt{\kappa H}$\,;
\item the cube $Q_{\ell}\subset \{{\rm dist}(x,\partial\Omega)>2\kp^{-1/2}\}$ and the function $\chi_\ell$ are as in Notation~\ref{not:Q-l}\,;
\item $\phi\in C^\infty(\overline{Q_\ell})$ is the function defined by Lemma~\ref{lem:gauge}\,;
\item $\Pi_{1}$ is the projection introduced  in \eqref{eq:Pi-3D}\,;
\item $x_j$ is the center of the cube $Q_\ell$ and
 $$v(x)=\Big(e^{i\kappa H\phi}\chi_{\ell}\psi\Big)\left(x_{j}+\frac{x}{\sqrt{\kp H}}\right),\qquad (x\in (-R/2,R/2)^{3}).$$
\end{itemize} 
\end{lemma}
 \begin{proof}
{\bf Step~1. Proof of \eqref{eq:op-l2*}.}

 By a gauge transformation and a translation, we may assume that the center of $Q_{\ell}$ is $x_{j}=0$. 
 We infer from \eqref{eq:lb-en-l4} that, for $\kappa$ sufficiently large, 
\[
\int_{Q_{\ell}}\left(|(\nb-i\kp H{\bf F})\chi_\ell\psi e^{i\kappa H\phi}|^{2}-\kp^{2} |\chi_\ell\psi e^{i\kappa H\phi}|^{2}        \right)dx<0.
\]
Performing the change of variables $x\mapsto \sqrt{\kp H}x$, we get 
\[
\int_{(-R/2,R/2)^3}\left(|(\nb-i{\bf F})v|^{2}-(1+\gamma)|v|^{2}        \right)dx<0,
\]
where $\gamma=\frac{\kp}{H}-1\approx 1-\frac{H}\kappa$. 
Now the estimate in \eqref{eq:op-l2*} follows simply by  applying Lemma~\ref{lem:op-3D}.

{\bf Step~2. Proof of \eqref{eq:op-l2**}.}

Using a change of variable, the min-max principle and \eqref{eq:3D-sp}, we get,
\begin{equation}\label{eq:proof-l2**}
\begin{aligned}
 \mathcal{E}_{0}(e^{i\kappa H\phi}\chi_{\ell}\psi,{\bf F };Q_{\ell})&
=\dfrac{1}{\sqrt{\kp H}}\int_{(-R/2,R/2)^3}\left(|(\nb-i{\bf F})v|^{2}-\frac{\kp}{H}|v|^{2}+\dfrac{\kp}{2H}|v|^{4}\right) dx\\
&\geq \dfrac{1}{\sqrt{\kp H}}\int_{(-R/2,R/2)^3}\left(\left(1-\frac{\kp}{H}\right)|\Pi_{1}v|^{2}+\dfrac{\kp}{2H}|v|^{4}\right) dx\,.
\end{aligned}
\end{equation}

{\bf Step~3. Proof of \eqref{eq:op-l4-truncated}.}
We perform the change of variable $x\mapsto x/\sqrt{\kappa H}$ to get
$$\frac1{R^3}\int_{(-R/2,R/2)^3}|v|^4\,dx=\frac1{\ell^3}\int_{Q_\ell}|\chi_\ell\psi|^4\,dx\,.$$
We use the estimate in Remark~\ref{rem:op-l4} coupled with H\"older's inequality  and our choice of $\ell$ to write
$$\int_{Q_\ell\setminus Q_{\ell-\frac1{\sqrt{\kappa H}}}}(\chi_\ell^4-1)|\psi|^4\,dx=
\left(1-\frac{H}\kappa\right)^2\ell^3o(1)\,.$$
Now, by Corollary~\ref{corol:lb-op-l4}, 
$$
\begin{aligned}
\int_{Q_\ell}|\chi_\ell\psi|^4\,dx&=\int_{Q_\ell}|\psi|^4\,dx+\int_{Q_\ell\setminus Q_{\ell-\frac1{\sqrt{\kappa H}}}}(\chi_\ell^4-1)|\psi|^4\,dx\\
&=-2E_{\rm Ab}\left(1-\frac{H}\kappa\right)^2\ell^3+\left(1-\frac{H}\kappa\right)^2\ell^3o(1)\,.
\end{aligned}
$$

{\bf Step~4. Proof of \eqref{eq:op-l2-lb**}.}

We use \eqref{eq:op-l4-truncated} and the following estimate from  Corollary~\ref{corol:lb-op-l4}
$$
\mathcal{E}_{0}(e^{i\kappa H\phi}\chi_{\ell}\psi,{\bf F };Q_{\ell})=E_{\rm Ab}[\kappa-H]^{2}\ell^3+[\kappa-H]^2\ell^3o(1)
$$
and infer from \eqref{eq:proof-l2**}
$$
-\int_{(-R/2,R/2)^3}|\Pi_1v|^2\leq 2E_{\rm Ab}\left(1-\frac{H}\kappa\right)R^3+R^3o(1)\,.
$$
This finishes the proof of \eqref{eq:op-l2-lb**} in light of the estimate in \eqref{eq:op-l2*}.
\end{proof}

             %
             %
 
\begin{thm}\label{corol:Pi-Ab}
There exist $\kappa_0>1$ and  and a function ${\rm err}:[\kappa_0,\infty)\to(0,\infty)$  such that:
\begin{itemize}
\item $\displaystyle\lim_{\kappa\to\infty}{\rm err}(\kappa)=0$\,;
\item the following  inequality  hold
 \begin{equation}\label{eq:op-l2*-main}
\left|\frac1{|Q_\ell|}\int_{Q_\ell}|\psi|^2\,dx-E_{\rm Ab}\left(1-\frac{H}\kappa\right)\right|\leq \left(1-\frac{H}\kappa\right){\rm err}(\kappa)\,,
 \end{equation}
\end{itemize}
where
\begin{itemize}
\item $\kappa\geq\kappa_0$ and $(\kappa,H)$ satisfy Assumption~\ref{assump'}\,;
\item $(\psi,\Ab)$ is a minimizer of  \eqref{GL-Energy}\,;
\item $\ell=(\kappa H)^{-1/2}\sqrt{2\pi[(\kappa-H)^{1/3}\kappa^{1/6}H]}$ with $[\cdot]$ denoting the integer part (floor function)\,;
\item the cube $Q_{\ell}\subset \{{\rm dist}(x,\partial\Omega)>2\kp^{-1/2}\}$ 
is as in Notation~\ref{not:Q-l}\,.
\end{itemize}
\end{thm}
\begin{proof}
We will prove \eqref{eq:op-l2*-main} in two steps by establishing the upper and lower bounds  in \eqref{eq:op-l2*-main} independently.

The lower bound follows easily from  Theorem~\ref{thm:Pi-Ab} used with  $R=\ell\sqrt{\kappa H}$ and $\ell$ as defined in Theorem~\ref{corol:Pi-Ab}. Namely we use \eqref{eq:op-l2-lb**}.


The proof of the upper bound is a bit lengthy. We introduce the  parameters
\begin{multline}\label{eq:ep-L-l'}
\alpha=\left(1-\frac{H}\kappa\right)^{1/16}\,,\quad \epsilon=\left(1-\frac{H}\kappa\right)^{3/8}\,,\quad L=\left(1-\frac{H}\kappa\right)^{-5/8}\,,\\
\quad \ell'=(\kappa-H)^{-1}\epsilon\quad{\rm and}\quad R'=\ell'\sqrt{\kappa H}\,.\end{multline}
Note that these parameters satisfy
\begin{equation}\label{eq:ep-L-l'*}
\left(1-\frac{H}{\kappa}\right)^2 R'^2L\ll 1\,,\quad \kappa^{-1}\ll\ell'\ll 1\quad{\rm and}\quad 1\ll R'\ll R\,,
\end{equation}
and
\begin{equation}\label{eq:ep-L-l'**}
\big((\ell')^{-2}+\kappa^2L^{-2}\big)\alpha^{-2}\left(1-\frac{H}\kappa\right)\ll \ell^3[\kappa-H]^2\,.
\end{equation}
Here
\begin{equation}\label{eq:R***}
R=\ell\sqrt{\kappa H}\,,
\end{equation}
and $\ell$ is defined in Theorem~\ref{corol:Pi-Ab}.

{\bf Step~1.} 

Let $(\widetilde{Q}_{\ell',L,i})_i$ be a family of $(\ell',\frac{L}{\sqrt{\kp H}})$-boxes covering the cube $Q_\ell$ (cf. Definition~\ref{defn:box}). 
These boxes are constructed as follows. First we cover $Q_\ell$ by $N$ boxes of the form
$$\widetilde Q_{\ell',\mathcal L,i}=\Big(-\frac{\ell'}{2},\frac{\ell'}{2}\Big)^2\times \Big(-\frac{L}{2\sqrt{\kp H}},\frac{L}{2\sqrt{\kp H}}\Big)+x_i\,,\quad x_i\in\R^3\,.$$
We choose these boxes to be disjoint (see Figure~\ref{fig}), hence the number $N$ satisfies
\begin{equation}\label{eq:Nb*}
\left| N-\frac{\ell^3\sqrt{\kappa H}}{\ell'^2L}\right|\leq C\frac{\ell^2\sqrt{\kappa H}}{\ell'^2L}\,.
\end{equation}  
Now we choose the boxes $Q_{\ell',L,i}$  by expanding the sides of $\widetilde Q_{\ell', L,i}$ slightly. Precisely, we take
$$ Q_{\ell',L,i}=\Big(-(1+\alpha)\frac{\ell'}{2},(1+\alpha)\frac{\ell'}{2}\Big)^2\times \Big(-(1+\alpha)\frac{L}{2\sqrt{\kp H}},(1+\alpha)\frac{L}{2\sqrt{\kp H}}\Big)+x_i\,.  $$

Consider a partition of unity $(h_i)$  satisfying in $Q_\ell$
$$\sum_{i}h_i=1\,,\quad \sum_i|\nabla h_i|^2\leq C\big((\ell')^{-2}+\kappa^2L^{-2}\big)\alpha^{-2}\,,$$
and ${\rm supp}~h_{i}\subset Q_{\ell',L,i}$.

Let the notation be as in Lemma~\ref{thm:Pi-Ab} and denote by 
\begin{equation}\label{eq:def-w}
w=e^{i\kappa H\phi}\chi_{\ell}\psi\,.
\end{equation}
 \begin{figure} \label{fig} 
 \centering
 \includegraphics[width=1\textwidth]{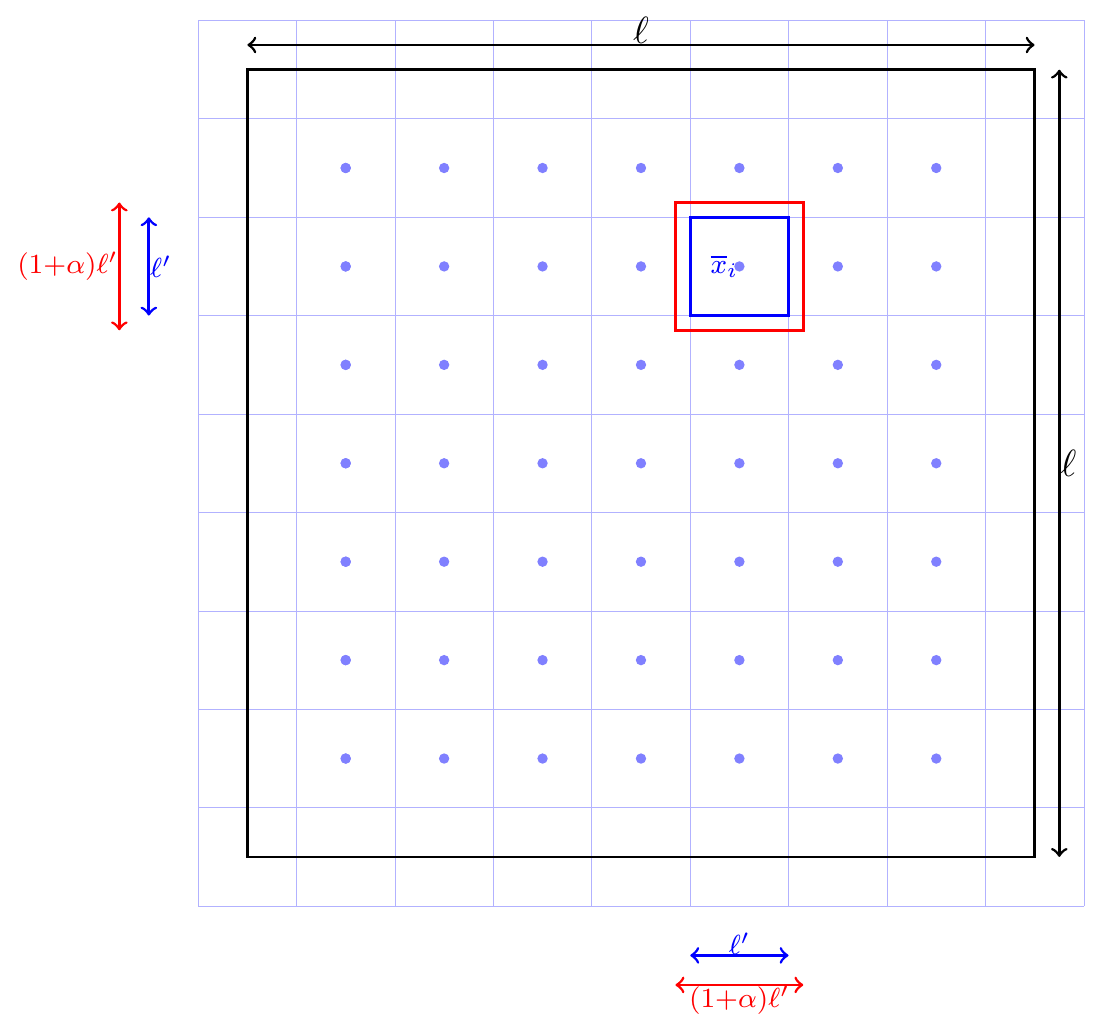}
\caption{The  projection on the $xy$-plane of the cube $Q_{\ell}$ decomposed into the small boxes $\widetilde Q_{\ell',L,i}$. Note the representation of the box $\widetilde Q_{\ell',L,i}$ with center $x_i=(\bar x_i,z_{i})\in\R^3$ and the slightly larger box $Q_{\ell',L,i}$.}
\end{figure}

We have the decomposition formula,
$$ 
\begin{aligned}
\mathcal{E}_{0}(w,{\bf F };Q_{\ell})
&\geq \sum_{i}\mathcal E_0(h_i w,\Fb;Q_{\ell',L,i})
-\sum_i\big\|\,|\nabla h_i|\psi\,\big\|_{L^2(Q_\ell)}^2\\
&\geq \sum_{i}\mathcal E_0(h_i w,\Fb;Q_{\ell',L,i})
-C\big((\ell')^{-2}+\kappa^2L^{-2}\big)\alpha^{-2}\|\psi\|_{L^2(Q_\ell)}^2\\
&\geq \sum_{i}\mathcal E_0(h_i w,\Fb;Q_{\ell',L,i})-C\big((\ell')^{-2}+\kappa^2L^{-2}\big)\alpha^{-2}\ell^3\left(1-\frac{H}\kappa\right)
\quad[{\rm by Remark~\ref{rem:op-l4}}]\\
&\geq \sum_{i}\mathcal E_0(h_i w,\Fb;Q_{\ell',L,i})-\ell^3[\kappa-H]^2o(1)\quad[{\rm by~}  \eqref{eq:ep-L-l'**}]\,.
\end{aligned}$$
In light of Corollary~\ref{corol:lb-op-l4}, we get
\begin{equation}\label{eq:step1-op-l2-main}
\sum_{i}\mathcal E_0(h_i w,\Fb;Q_{\ell',L,i})\leq E_{\rm Ab}(\kappa-H)^2\ell^3+\ell^3(\kappa-H)^2o(1)\,.
\end{equation}
%

{\bf Step~2.}

Let
\begin{equation}\label{eq:form-qi}
q(h_i w,\Fb;Q_{\ell',L,i})=\int_{Q_{\ell',L,i}}\Big(|(\nabla-i\kappa H\Fb)h_{i}w|^2-\kappa^2|h_{i}w|^2\Big)\,dx\,.
\end{equation}
We introduce the two sets of indices
$$\mathcal J_+=\{i~:~q(h_i w,\Fb;Q_{\ell',L,i})> 0\}\quad{\rm and}\quad
\mathcal J_-=\{i~:~q(h_i w,\Fb;Q_{\ell',L,i})\leq 0\}\,.$$
Let $N_+={\rm Card},\mathcal J_+$ and $N_-={\rm Card}\,\mathcal J_-$. We will prove 
that
\begin{equation}\label{eq:step2-op-l2-main}
 \left| N_- -\frac{\ell^3\sqrt{\kappa H}}{\ell'^2L}\right| \leq \frac{\ell^3\sqrt{\kappa H}}{\ell'^2L}\,o(1)\,,
\end{equation}
and
\begin{equation}\label{eq:step2-op-l2-main*}
N_+=N_-o(1)\,.
\end{equation}
Since $N_++N_-=N$, \eqref{eq:step2-op-l2-main*} is a simple consequence of \eqref{eq:Nb*} and \eqref{eq:step2-op-l2-main}. The upper bound in \eqref{eq:step2-op-l2-main} is a simple consequence of \eqref{eq:Nb*} since $N_-\leq N$.

We turn to the proof of the lower bound in \eqref{eq:step2-op-l2-main}. We have the trivial lower bound that follows from \eqref{eq:m0gb} and \eqref{eq:en-m0}, 
Theorem~\ref{thm:Ab} and a change of variables
$$\mathcal E_0(h_i w,\Fb;Q_{\ell',L,i})\geq \dfrac{\kp}{H\sqrt{\kp H}} \int_{-L/2}^{L/2}m_0\left(\frac{H}\kappa,R'\right)\,dx_3\geq  \dfrac{\kp}{H\sqrt{\kp H}}  g\left(\frac{H}\kappa\right)R'^2L\,.$$
Here $R'=\ell'\sqrt{\kappa H}$. Using Theorem~\ref{thm:Ab}, we get further
$$\mathcal E_0(h_i w,\Fb;Q_{\ell',L,i})\geq  \Big(E_{\rm Ab}+o(1)\Big)\left(\kappa-H\right)^2\ell'^2L(\kappa H)^{-1/2}\,.$$
Inserting this into \eqref{eq:step1-op-l2-main} and dropping the positive terms corresponding to $i\in\mathcal J_+$, we get
$$
\begin{aligned}
N_-\Big(E_{\rm Ab}+o(1)\Big)\left(\kappa-H\right)^2\ell'^2L(\kappa H)^{-1/2}
&\leq \sum_{i\in\mathcal J_-}\mathcal E_0(h_i w,\Fb;Q_{\ell',L,i})\\
&\leq E_{\rm Ab}(\kappa-H)^2\ell^3+\ell^3(\kappa-H)^2o(1)\,.\end{aligned}$$
Since $E_{\rm Ab}<0$, this yields \eqref{eq:step2-op-l2-main}.

{\bf Step~3.}

We denote by $x_i$ the center of the box $Q_{\ell',L,i}$. If $i\in\mathcal J_-$, the change of the variable $x\mapsto (x-x_i)\sqrt{\kappa H}$ yields
$$\int_{Q_{R',L}}\Big(|(\nabla-i\Fb)v_i|^2-(1+\gamma)|v_i|^2\Big)\,dx\leq 0\,,$$
where $\gamma=1-\frac{\kappa}{H}$, $Q_{R',L}=(-R'/2,R'/2)^2\times(-L/2,L/2)$ and
\begin{equation}\label{eq:vi**}
v_i(x)=h_i w\left(x_i+\frac{x}{\sqrt{\kappa H}}\right)\,.\end{equation}
We apply Lemma~\ref{lem:op-3D} to obtain
\begin{equation}\label{eq:step3-main}
\|v_i-\Pi_1v_i\|_{L^p(Q_{R',L})}\leq C\sqrt{1-\frac{H}\kappa}\,\|v_i\|_{L^2(Q_{R',L})}\,,\quad p\in\{2,4\}\,,
\end{equation}
where $\Pi_1$ is the projection in \eqref{eq:Pi-3D}.
For $p=4$, we write by H\"older's inequality,
\begin{equation}\label{eq:step3-main-p=4}
\|v_i-\Pi_1v_i\|_{L^4(Q_{R',L})}\leq C(R'^2L)^{1/4}\sqrt{1-\frac{H}\kappa}\,\|v_i\|_{L^4(Q_{R',L})}\ll \|v_i\|_{L^4(Q_{R',L})}\,,
\end{equation}
by \eqref{eq:ep-L-l'*}. 
   Let us introduce the function $u_i$ as follows,
       \begin{equation}\label{eq:step3-ui}
v_i=\left(1-\frac{H}\kappa\right)^{1/2}u_i.
       \end{equation}
Since $R'\gg1 $, we get  (cf. \eqref{cR} and Theorem~\ref{thm:Ab}) 
$$\int_{Q_{R',L}}\Big(-|\Pi_1u_i|^2+\frac12|\Pi_1u_i|^4\Big)\,dx\geq \int_{-L/2}^{L/2}c(R')\,dx_3\geq E_{\rm Ab}R'^2L-R'^2Lo(1)\,.$$
Thus, we get,
$$
-\sum_{i\in\mathcal J_-}\int_{Q_{R',L}}|\Pi_1u_i|^2\,dx
\geq -\frac12 \sum_{i\in\mathcal J_-}\int_{Q_{R',L}}|\Pi_1u_i|^4\,dx+\Big(E_{\rm Ab}+o(1)\Big)R'^2L N_-\,.$$
Using \eqref{eq:step3-main-p=4}, \eqref{eq:step2-op-l2-main} and $R'=\ell'\sqrt{\kappa H}$, we get further
\begin{equation}\label{eq:step3-main*}
-\sum_{i\in\mathcal J_-}\int_{Q_{R',L}}|\Pi_1u_i|^2\,dx\geq-\frac12\big(1+o(1)\big) \sum_{i\in\mathcal J_-}\int_{Q_{R',L}}|u_i|^4\,dx
+E_{\rm Ab}\ell^3(\kappa H)^{3/2} +\ell^3(\kappa H)^{3/2}o(1)\,.
\end{equation}
In light of \eqref{eq:step3-ui}, \eqref{eq:vi**} and \eqref{eq:def-w}, we get by a change variable transformation
$$
\begin{aligned}
\sum_{i\in\mathcal J-}\int_{Q_{R',L}}|u_i|^4\,dx
&=(\kappa H)^{3/2}\left(1-\frac{H}\kappa\right)^{-2}\sum_{i\in\mathcal J-}\int_{Q_{\ell^{\prime},\tiny\frac{L}{\sqrt{\kp H}}}}|h_iw|^4\,dx\\
&
\leq (\kappa H)^{3/2}\left(1-\frac{H}\kappa\right)^{-2}\int_{Q_{\ell}}|\psi|^4\,dx
\leq -2E_{\rm Ab}(\kappa H)^{3/2}\ell^3+(\kappa H)^{3/2}\ell^3o(1)
\end{aligned}
$$
by Corollary~\ref{corol:lb-op-l4}. Inserting this into \eqref{eq:step3-main*}, we get
$$
-\sum_{i\in\mathcal J_-}\int_{Q_{R',L}}|\Pi_1u_i|^2\,dx\geq 
2E_{\rm Ab}\ell^3(\kappa H)^{3/2} +\ell^3(\kappa H)^{3/2} o(1)\,.
$$
Now, using \eqref{eq:step3-main}, we may write,
\begin{equation}\label{eq:step3-main***}
\sum_{i\in\mathcal J_-}\int_{Q_{R',L}}|u_i|^2\,dx\leq 
\big(1+o(1)\big)\sum_{i\in\mathcal J_-}\int_{Q_{R',L}}|\Pi_1u_i|^2\,dx\leq
-2E_{\rm Ab}\ell^3(\kappa H)^{3/2} +\ell^3(\kappa H)^{3/2} o(1)\,.
\end{equation}
Recall the expression of $u_i$ in \eqref{eq:step3-ui}.  Performing  a change of variable, we get
$$
\int_{Q_{R',L}}|u_i|^2\,dx=\left(1-\frac{H}\kappa\right)^{-1}(\kappa H)^{3/2}
\int_{Q_{\ell^{\prime},\tiny\frac{L}{\sqrt{\kp H}}}}|h_i\chi_\ell\psi|^2\,dx\,.$$
Using \eqref{eq:step2-op-l2-main} and \eqref{eq:step2-op-l2-main*}, we get
$$\begin{aligned}
\sum_{i\in\mathcal J_-}\int_{Q_{R',L}}|u_i|^2\,dx
&=\left(1-\frac{H}\kappa\right)^{-1}(\kappa H)^{3/2}\sum_{j\in\mathcal J_\pm}\int_{Q_{\ell^{\prime},\tiny\frac{L}{\sqrt{\kp H}}}}|h_i\chi_\ell\psi|^2\,dx+o\left(\frac{\ell^3\sqrt{\kappa H}}{\ell'^2L}\right)\\
&=\left(1-\frac{H}\kappa\right)^{-1}(\kappa H)^{3/2}\int_{Q_\ell}|\chi_\ell\psi|^2\,dx+\ell^3(\kappa H)^{3/2} o(1)\,,\end{aligned}
$$
by the definition of $\ell'$ and $L$ in \eqref{eq:ep-L-l'}.
We insert this into \eqref{eq:step3-main***} and get,
\begin{equation}\label{eq:step3-op-l2-main}
\int_{Q_\ell}|\chi_\ell\psi|^2\,dx\leq -2E_{\rm Ab}\ell^3\left(1-\frac{H}\kappa\right)+\ell^3\left(1-\frac{H}\kappa\right)o(1)\,.
\end{equation}
The estimate in Remark~\ref{rem:op-l4} and H\"older's inequality yield
$$\int_{Q_\ell}(1-\chi_\ell^2)|\psi|^2\,dx\leq 
\int_{Q_\ell\setminus Q_{\ell-\frac1{\sqrt{\kappa H}}}}|\psi|^2\,dx\leq \frac{\ell^{5/2}}{(\kappa H)^{1/4}}\left(1-\frac{H}\kappa\right)=o\Bigg(\ell^3\left(1-\frac{H}{\kappa}\right)\Bigg)\,.$$
Inserting this into \eqref{eq:step3-op-l2-main}, we get 
 the upper bound in Theorem~\ref{corol:Pi-Ab}.
\end{proof}

\subsection*{Acknowledgments}
AK is supported by a research grant from Lebanese University.

\end{document}